\documentclass[10pt]{amsart}
\usepackage{amsmath, amssymb,url}
\usepackage{amsfonts,amscd}
\usepackage[a4paper, margin=1.25in]{geometry}
\usepackage{parskip}
\usepackage{color}
\allowdisplaybreaks

\newtheorem{theorem}{Theorem}[section]

\newtheorem{prop}[theorem]{Proposition}
\renewenvironment{proof}{\par\noindent{\bf Proof.}}{$\square$\par\bigskip}
\newtheorem{lemma}[theorem]{Lemma}

\newtheorem{cor}[theorem]{Corollary}

\newtheorem{thm}{Theorem}[section]
\newtheorem{question}[thm]{Question}

\def\Z{\mathbb Z}
\def\N{\mathbb N}

\def\R{\mathbb R}
\def\C{\mathbb C}

\def\O{\operatorname{O}}
\def\o{\operatorname{o}}
\def\log{\operatorname{log}}

\def\p{\operatorname{\preccurlyeq}}

\def\thtp{\operatorname{\theta_{\textit{f}}\,(\textit{p})}}

\def\p{\operatorname{\widehat{\phi_{\it L}}} }

\definecolor{blue}{rgb}{.2,.6,.75}
\definecolor{green}{rgb}{.4,.7,.4}
\definecolor{purple}{RGB}{127,0,255}
\begin{document}

\title[Central limit theorems for elliptic curves and modular forms]{Central limit theorems for elliptic curves and modular forms with smooth weight functions}
\author[Stephan Baier]{Stephan Baier}
\date{\today}
\address{Stephan Baier, Ramakrishna Mission Vivekananda University, Department of Mathematics, PO Belur Math, Distt Howrah - 711202, West Bengal, India}
\email{stephanbaier2017@gmail.com}
\author[Neha Prabhu]{Neha Prabhu}
\address{Neha Prabhu, The Institute of Mathematical Sciences, C.I.T. Campus, Taramani, Chennai - 600113, Tamil Nadu, India}
\email{nehap@imsc.res.in}
\author[Kaneenika Sinha]{Kaneenika Sinha}
\address{Kaneenika Sinha, IISER Pune, Dr Homi Bhabha Road, Pashan, Pune - 411008, Maharashtra, India}
\email{kaneenika@iiserpune.ac.in}

\begin{abstract} In \cite{PrabhuSinha}, the second and third-named authors established a Central Limit Theorem for the error term in the 
Sato-Tate law for families of modular forms. This method was adapted to families of 
elliptic curves in \cite{BP19} by the first and second-named authors. In this context, a Central Limit Theorem was established only 
under a strong hypothesis going beyond the Riemann Hypothesis. 
In the present paper, we consider a smoothed version of the Sato-Tate conjecture, which allows us to overcome several limitations. In particular,
for the smoothed version, we are able to establish a Central Limit Theorem for much smaller families of modular forms, and we succeed in 
proving a theorem of this type for families of elliptic curves under the Riemann Hypothesis for $L$-functions associated to Hecke eigenforms 
for the full modular group. 
\end{abstract}

\maketitle

\bigskip

\section{Introduction}
 A sequence $X = \{x_n\}$ of real numbers in $[0,1]$ is said to be equidistributed with respect to a probability measure $\mu$ (or $\mu$-equidistributed) if for all $0 \leq \alpha \leq \beta \leq 1,$
\begin{equation}\label{ed-1}
\lim_{V \to \infty} \frac{1}{V} \#\{ 1 \leq n \leq V:\,\alpha \leq x_n \leq \beta \} = \int_{[\alpha,\beta]}d\mu.
\end{equation}
Equivalently, for any continuous function $\phi:\,[0,1] \to \C,$
\begin{equation}\label{ed-2}
\lim_{V \to \infty} \frac{1}{V} \sum_{n=1}^V \phi(x_n) = \int_{[0,1]} \phi d\mu.
\end{equation}
A pertinent analytic question about a $\mu$-equidistributed sequence is the rate of convergence in \eqref{ed-1} and \eqref{ed-2}.  That is,

\begin{question}\label{discrepancy}
Can we find explicit bounds for the discrepancies
$$D_X(V) := \left| \frac{1}{V} \#\{ 1 \leq n \leq V:\,\alpha \leq x_n \leq \beta \} -\int_{[\alpha,\beta]}d\mu\right|$$
and
$$D_X(\phi,V) := \left| \frac{1}{V}\sum_{n=1}^V \phi(x_n) - \int_{[0,1]} \phi d\mu\right|$$
in terms of $V?$
\end{question}
Other related questions to these discrepancies are stated below.
\begin{question}\label{averages}
By varying the sequences $X$ in a suitable family $\mathcal F$, 
can we obtain average error terms
$$\frac{1}{|\mathcal F|}\sum_{X \in \mathcal F} D_X(V)$$
and
$$\frac{1}{|\mathcal F|}\sum_{X \in \mathcal F} D_X(\phi,V)?$$
Does the order of the discrepancy improve upon averaging?
\end{question}

\begin{question}\label{fluctuations}
How do the terms $D_X(V)$ and $D_X(\phi,V)$ fluctuate as one varies the sequences $X$ in a family $\mathcal F$? In the case of $D_X(\phi,V)$, how are the fluctuations affected by the choice of the functions $\phi$?

\end{question}
While the template for the study of the above stated questions is provided by Fourier analysis, the answers arise out of arithmetic properties of the concerned sequences.  In this article, we focus on two sequences which are equidistributed with respect to the Sato-Tate measure $$\mu(t) := 2\sin^2 \pi t.$$We address Questions \ref{discrepancy}, \ref{averages} and \ref{fluctuations} for these sequences. 

{\bf Sequences arising from elliptic curves:}
Let $E(a,b)$ denote an elliptic curve given by the equation
$$y^2 = x^3 + ax + b,$$
where $a,\,b \in \Z$ and $\Delta(a,b) := 4a^3 + 27b^2 \neq 0.$  

Let $\mathcal N_E$ denote the conductor of the elliptic curve $E$ and $L(E;s)$ denote the Hasse-Weil $L$-function 
$$L(E;s) := \sum_{n=1}^{\infty} \frac{a_E(n)}{n^s} = \prod_{p \mid \mathcal N_E}\left(1 -   \frac{a_E(p)}{p^s}\right)^{-1}\prod_{p \nmid \mathcal N_E}\left(1 -   \frac{a_E(p)}{p^s} + p^{1 - 2s}\right)^{-1},\,\text{ Re }(s) > \frac{3}{2}.$$
We denote 
$$\tilde{a}_E(n) = \frac{a_E(n)}{\sqrt n}.$$
By a classical theorem of Hasse, if $p$ is a prime such that $p \nmid \mathcal N_E,$ we have $\tilde{a}_E(p) \in [-2,2].$  Thus, we write $\tilde{a}_E(p) = 2 \cos \pi \theta_E(p)$ for $\theta_E(p) \in [0,1].$  In the 1960s, Sato and Tate independently conjectured that for an elliptic curve as defined above, the sequence $\{\theta_E(p)\}_{p\text{ prime}}$ is equidistributed in $[0,1]$ with respect to the measure $\mu(t)dt,$ where $\mu(t) = 2 \sin^2 \pi t.$  That is,  for any interval $I = [\alpha,\beta] \subset [0,1],$
\begin{equation}\label{ST-E}
\lim_{x \to \infty} \frac{\#\{p \leq x:\,(p,\mathcal N_E) = 1,\,\theta_E(p) \in I\}}{\pi(x)}  = \int_I\mu(t) dt.
\end{equation}
Here, $\pi(x)$ denotes the number of primes less than or equal to $x.$

The Sato-Tate conjecture is now a theorem by the work of Clozel, Harris, Shepherd-Barron and Taylor (\cite{CHT}, \cite{HST} and \cite{Taylor}).  
Henceforth, for an interval $I = [\alpha,\beta] \subset [0,1],$ define
 $$M_I(E,x) = \#\{p \leq x :\, p\text{ prime},\,(p,\mathcal N_E) = 1,\,\theta_E(p) \in I\}. $$

In the spirit of Question \ref{discrepancy}, what can we say about error terms in the Sato-Tate distribution law?  In this respect, Murty \cite{VKMurty} obtained conditional effective error terms for \eqref{ST-E}.  Under the assumption that all symmetric power $L$-functions of $E(a,b)$  can be analytically continued to $\C,$  have suitable functional equations and satisfy the (Generalized) Riemann Hypothesis, he showed that
\begin{equation}\label{Kumar-Murty}
M_I(E,x) = \pi(x)\int_I \mu(t) dt + \O\left(x^{3/4}\sqrt{\log \mathcal N_Ex }\right).
\end{equation}
The Riemann Hypothesis for the symmetric power $L$-functions of $L(E(a,b);s)$ plays a vital role in the study of the error term $M_I(E(a,b),x) - \pi(x)\int_I \mu(t) dt$ and currently, no estimates are available  without this assumption.  

With respect to Question \ref{averages}, one can investigate the error terms in $M_I(E,x)$ by averaging over suitable families of elliptic curves.  One may consider an associated quantity,
$$N_I(E(a,b),x) = \#\{x/2 < p \leq x:\,p\text{ prime},\,(p,\mathcal N_E) = 1,\,\theta_E(p) \in I\}$$ (The restriction of $p$ to dyadic intervals turns out useful in our later treatment.)
In this direction, the first-named author of this article and Zhao (\cite[Theorem 1]{BZ}) have obtained results from which estimates for the average of the error term
$$\frac{1}{4AB}\sum_{|a| \leq A}\sum_{|b| \leq B}\left(N_I(E(a,b),x) - \tilde{\pi}(x) \mu(I)\right)$$
follow if $A,B$ and the length of the interval $I$ satisfy suitable conditions. Here,
$$\tilde{\pi}(x) := \#\{x/2 < p \leq x,\,p \text { prime }\} \text{ and }\mu(I) := \int_I\mu(t)dt. $$
\cite{BZ} contains estimates for the second moment
$$\frac{1}{4AB}\sum_{|a| \leq A}\sum_{|b| \leq B}\left(N_I(E(a,b),x) - \tilde{\pi}(x) \mu(I)\right)^2.$$
The above average results do not require the assumptions made in \cite{VKMurty}.  We now describe another (related) sequence which follows the Sato-Tate distribution law.

{\bf Sequences arising from modular forms:}
The Sato-Tate conjecture was generalized by Serre in the context of modular forms.  Indeed, for a positive, even integer $k$ and a positive integer $N,$ let $S(N,k)$ denote the space of modular cusp forms of weight $k$ with respect to $\Gamma_0(N).$  For $n \geq 1,$ let $T_n$ denote the $n$-th Hecke operator acting on $S(N,k).$  We denote the set of Hecke newforms  in $S(N,k)$  by $\mathcal F_{N,k}.$  Any $f(z) \in \mathcal F_{N,k}$ has a Fourier expansion
$$f(z) = \sum_{n=1}^{\infty} {n^{\frac{k-1}{2}}}a_f(n) q^n, \qquad q = e^{2\pi i z},$$
where 
$a_f(1) = 1$ and
$$\frac{T_n(f(z))}{n^{\frac{k-1}{2}}} = a_f(n) f(z),\,n \geq 1.$$
We consider a newform $f(z)$ in $\mathcal F_{N,k}.$  Let $p$ be a prime number with $(p,N) = 1.$  By a theorem of Deligne, the eigenvalues $a_f(p)$ lie in the interval $[-2,2].$  Denoting $a_f(p) = 2\cos \pi \thtp,\,$ with $ \thtp \in [0,1],$ the Sato-Tate conjecture for modular forms is the assertion that if $f$ is a non-CM  newform in $\mathcal F_{N,k},$ the sequence $\{\thtp\}_{p \text{ prime}}$ is equidistributed in the interval $[0,1]$ with respect to the Sato-Tate measure $\mu(t) dt.$  That is, for any interval $I \subset [0,1],$
\begin{equation}\label{ST-M}
\lim_{x \to \infty}\frac{1}{\pi(x)}\#\{p \leq x :\, (p,N) = 1,\,\thtp \in I \} = \int_I \mu(t) dt.
\end{equation}  
By the modularity theorem of Wiles \cite{Wiles}, for an elliptic curve $E = E(a,b)$ as defined above, there exists a newform $f \in \mathcal F_{\mathcal N_E,2}$ with rational Fourier coefficients such that $L(E;s)$ is equal to the $L$-function of $f,$
$$L_f(s) := \sum_{n=1}^{\infty} \frac{n^{1/2}a_f(n)}{n^s},\,\text{ Re }(s) > \frac{3}{2}.$$     Thus, the Sato-Tate conjecture \eqref{ST-M} for modular forms is a generalization of the Sato-Tate conjecture \eqref{ST-E} for elliptic curves.   This (general) conjecture is now a theorem by the work of Barnet-Lamb, Geraghty, Harris and Taylor \cite{BGHT}. 

Henceforth, for an interval $I = [\alpha,\beta] \subset [0,1],$ define, for $f \in \mathcal F_{N,k},$
 
$$N_I(f,x) = \#\{p \leq x :\, p\text{ prime},\,(p,N) = 1,\,\thtp \in I\}.$$
Recently, under the assumption of similar analytic hypotheses for symmetric power $L$-functions for $f \in \mathcal F_{N,k}$ as those made in \cite{VKMurty},  Rouse and Thorner \cite{RT} sharpened and generalized the error term in Murty's result \eqref{Kumar-Murty} to all even $k \geq 2.$  They showed that under the aforementioned analytic hypotheses, for all $k \geq 2$ and squarefree $N,$ 
$$N_I(f,x) =  \pi(x)\int_I \mu(t) dt + \O\left(x^{3/4}\right)$$
for any $f \in \mathcal F_{N,k}.$
By the work of Wang \cite{Wang}, we now have the following unconditional average estimate for the error terms of $N_I(f,x)\,f \in \mathcal F_{1,k}.$
$$\frac{1}{|\mathcal F_{1,k}|}\sum_{f \in \mathcal F_{1,k}}\left(N_I(f,x) - \pi(x) \mu(I)\right) = \O\left(\log \log x + \frac{\log x}{k}\right).$$

We now make some remarks about Question \ref{fluctuations} with respect to both types of sequences described above.

In \cite{PrabhuSinha}, the second and third-named authors of this article considered the following perspective. 
One may view $\mathcal F_{N,k}$ as a finite measure space with the uniform distribution of mass $1/|\mathcal F_{N,k}|$ to each element $f$ in the space.  
Now let the random variable $X_k:\,\mathcal F_{N,k} \to \Z_{\geq 0}$ be given by $f \to N_I(f,x).$  We can then evaluate the asymptotic behaviour of expected value of this 
random variable, the variance and higher moments of suitable normalizations of this random variable.   More precisely, \cite{PrabhuSinha} contains the following theorem.

\begin{thm} \label{moment}
Let $N \geq 1.$  Suppose $k = k(x)$ satisfies $\frac{\log k}{\sqrt x \log x} \to \infty$ as $x \to \infty.$  Then,
for any integer $r \geq 1,$
$$\lim_{x \to \infty}\frac{1}{|\mathcal F_{N,k}|}\sum_{f \in \mathcal F_{N,k}}\left(\frac{\left(N_I(f,x) - \pi(x) \mu(I)\right)}{\sqrt{\pi(x)\left(\mu(I) - \mu(I)^2\right)}}\right)^r = 
\begin{cases}
0 &\text{ if }r\text{ is odd}\\
\frac{r!}{\left(\frac{r}{2}\right)!2^{r/2}}&\text{ if }r\text{ is even.}
\end{cases}$$
\end{thm}

Since
$$\int_{-\infty}^{\infty} t^r e^{\frac{-t^2}{2}}dt = \begin{cases}
0 &\text{ if }r\text{ is odd}\\
\frac{r!}{\left(\frac{r}{2}\right)!2^{r/2}}&\text{ if }r\text{ is even,}
\end{cases}$$
the above theorem implies that the $r$-th moments of the normalized error
$$\frac{\left(N_I(f,x) - \pi(x) \mu(I)\right)}{\sqrt{\pi(x)\left(\mu(I) - \mu(I)^2\right)}}$$
converge to those of the Gaussian distribution as $x \to \infty$ under the growth conditions for the weights $k=k(x)$ in Theorem \ref{moment}
above. 
Therefore, Theorem \ref{moment} can be viewed as a version of the Central Limit Theorem for the said normalized error.

Adapting the methods of \cite{PrabhuSinha} to families of elliptic curves, the first and second-named authors \cite{BP19} proved an elliptic curve analogue of the above 
theorem under a very strong hypothesis going beyond the Riemann Hypothesis. Unconditionally and under the Riemann Hypothesis for 
$L$-functions associated to 
Hecke eigenforms (which we denote as ``MRH'', meaning ``Modular Riemann Hypothesis''), 
they managed to obtain new upper bounds for the moments in question but failed
to establish asymptotic estimates. However, under MRH, 
their upper bounds differ from the expected sizes of
the moments only by powers of logarithm of $x$. 
In \cite{BP19}, the connection to Hecke eigenforms comes from work of Birch \cite{Birch}, who expressed moments of the coefficients $a_E(p)$ in terms of 
traces of Hecke operators for the full modular group.  After a long chain of calculations, the problem boils down to averaging these traces of Hecke operators
over primes, which in turn can be reduced to averaging Fourier coefficients of individual Hecke eigenforms over primes. These averages are then
estimated by the prime number theorem for Hecke eigenforms. This is the point where MRH comes into play: We obtain significantly better bounds 
(essentially square root cancellation) under the Riemann Hypothesis for the corresponding $L$-functions than unconditionally. This saves us a 
large amount when it comes to estimating the quantity in Lemma \ref{Save}.  

The main idea that goes into the proof of Theorem \ref{moment} (and the corresponding result for elliptic curves) 
is to approximate the characteristic functions $\chi_I(t)$ by smooth trigonometric polynomials 
$S(\theta) = \sum_{|n| \leq M} \widehat{S}(n) e(n\theta)$ defined by Beurling and Selberg \cite[see Chapter 1]{Mont} and therefore, approximate 
$$N_I(f,x) = \sum_{p \leq x} \chi_{I}(\thtp)$$
by the corresponding sums
$$\sum_{|n| \leq M} \widehat{S}(n)\sum_{  p \leq x} e(n \theta_f(p))$$ over the relevant families.  
The approximation of the characteristic function by these trigonometric polynomials gives rise to (essentially) two dominant error terms involving the parameter $M$. The process of choosing an optimal value for $M$ that balances these two error terms is what results in the growth condition of the weights $k$ with respect to $x$ in the case of modular forms. The case of elliptic curves is more complicated and the optimal choice of $M$ varies with the hypothesis assumed (unconditional, on MRH and on conditions stronger than MRH). The resulting combined error in the unconditional case and on assuming MRH however, exceeds the expected main term. Thus a central limit theorem is not obtained, even on assuming MRH.

In this article, we consider two types of smooth, periodic test functions $\phi$ and treat the associated sums
$$\sum_{p \leq x} \phi(\thtp),\,f \in \mathcal F_{N,k} \text{ instead of }\sum_{p \leq x} \chi_I(\thtp).$$
These smooth test functions are constructed in such a way that their Fourier series have rapid decay or are even finite. 
Therefore, cutting them off at a suitable, not too large $M$ will result in an error term that is negligible or even vanishes. This is 
a big advantage over taking the characteristic function of an interval and results in error terms that are easier to handle. In particular,
we manage to obtain a full analogue of Theorem \ref{moment} for elliptic curves under 
MRH and a strong improvement of Theorem \ref{moment} for modular forms with a much 
weaker growth condition on the weights $k=k(x)$. We now state the two main theorems of this article.

\subsection{New theorems with smooth weight functions}
\begin{thm}\label{main-elliptic}
 Let $\Phi \in C^{\infty}(\R)$ be a real-valued, even function in the Schwartz class and $\widehat{\Phi}$ its Fourier transform. Suppose that
 $\Phi(t) \ll (1+|t|)^{-2}$ and, for some fixed $\lambda,\omega>0$,    
 $\widehat{\Phi}(t) \ll e^{-\lambda |t|^{\omega}}$, as $|t| \to \infty$.  Fix a real number
 $L\ge 1$ and define 
$$\phi_L(t) = \sum_{m \in \Z} \Phi(L(t+m)) \text{ and }N_{\Phi,L,E}(x) = \sum_{\substack{x/2 < p \leq x \\ p \nmid \mathcal{N}_E}} \phi_L(\theta_E(p)).$$
Define
$$V_{\Phi,L} = \int_0^1\phi_L(t)^2\mu(t)dt - \left( \int_0^1 \phi_L(t) \mu(t) dt \right)^2.$$
%\frac{1}{L^2} \sum_{m \geq 1}\left(\widehat{\Phi}\left(\frac{m}{L}\right) - \widehat{\Phi}\left(\frac{m+1}{L}\right)\right)^2.$$  
Suppose $A = A(x) \geq 1,\,B = B(x) \geq 1$ satisfy $\frac{\log A}{\log x},\frac{\log B}{\log x}\rightarrow \infty$, as $x\rightarrow \infty$, 
and $\log(2AB)\le x^{1/2-\varepsilon}$. 
If the Riemann Hypothesis holds for all $L$-functions associated to Hecke eigenforms with respect to the full modular group, then
for any integer $r \geq 0,$
\begin{equation} \label{ellmoment}
\lim_{x \to \infty}\frac{1}{4AB}\sum_{|a| \leq A}\sum_{|b| \leq B} \left(\frac{N_{\Phi,L,E(a,b)}(x) - 
\tilde{\pi}(x)\int_0^1\phi_L(t) \mu(t)dt}{\sqrt{\tilde{\pi}(x)V_{\Phi,L}}}\right)^r = 
\begin{cases}
0 &\text{ if }r\text{ is odd}\\
\frac{r!}{\left(\frac{r}{2}\right)!2^{r/2}}&\text{ if }r\text{ is even.}
\end{cases}
\end{equation}
\end{thm}

\begin{thm}\label{main-modular}
Let $\Phi \in C^{\infty}(\R)$ be a real-valued, even function in the Schwartz class.  
Let $\phi_L(t)$ and $V_{\Phi,L}$ be as defined in Theorem \ref{main-elliptic}.  For $f \in \mathcal F_{N,k},$ define
$$N_{\Phi,L,f}(x) = \sum_{p \leq x \atop {p \nmid N}} \phi_L(\thtp).$$
\begin{enumerate}
\item[{\bf(a)}] Suppose $\widehat{\Phi}$ is compactly supported and $k = k(x) \geq 2$ satisfies $\frac{\log k}{\log x} \to \infty$ as $x \to \infty.$  Then,
for any integer $r \geq 0,$
\begin{equation}\label{Gaussian}
\lim_{x \to \infty}\frac{1}{|\mathcal F_{N,k}|}\sum_{f \in \mathcal F_{N,k}} \left(\frac{N_{\Phi,L,f}(x) - {\pi}(x)\int_0^1 \phi_L(t) \mu(t)dt}{\sqrt{{\pi}(x)V_{\Phi,L}}}\right)^r = 
\begin{cases}
0 &\text{ if }r\text{ is odd}\\
\frac{r!}{\left(\frac{r}{2}\right)!2^{r/2}}&\text{ if }r\text{ is even.}
\end{cases}
\end{equation}

\item[{\bf(b)}] For fixed $\lambda, \omega>0$, suppose the Fourier transform $\widehat{\Phi}$ satisfies $\widehat{\Phi}(t) \ll e^{-\lambda |t|^{\omega}}$, as $|t| \to \infty.$  Then, the asymptotic \eqref{Gaussian} holds if $k = k(x) \geq 2$ satisfies $\frac{(\log k)}{(\log x)^{1+1/\omega}} \to \infty$ as $x \to \infty.$
  
\end{enumerate}
\end{thm}
$ $\\
$ $\\
{\bf Remark 1:} It makes perfect sense to exclude the cases when $a=0$ or $b=0$ from the moments in \eqref{ellmoment} because those curves all have CM (recall that the Sato-Tate law is valid
only under the condition that $E(a,b)$ is non-CM). The proportion of the remaining CM-curves is small, so there is no loss in keeping them (for details, see \cite{FoMu}). \\  \\
{\bf Remark 2:} In the case of elliptic curves, it turns out that the assumption of MRH is essential in our method. Unconditionally, we have
not been able to make significant progress on the problem. \\ \\
{\bf Remark 3:} The periodic test functions $\phi_L(t)$ considered by us are centered at the integers and become very small if $||t||$, 
the distance of $t$ to the nearest integer, is much larger than $1/L$. Hence, $\phi_L(t)$, considered modulo 1, may be viewed as a smooth 
analogue of the 
characteristic function of the interval $[-1/L,1/L]$. It is possible to extend our method to functions centered at shifts $c+\mathbb{Z}$ of the
integers. That is, we can consider functions of the form
$$\phi_{L,c}(t) := \sum_{n \in \Z} \Phi(L(t+n +c)),\,0 < c < 1.$$
However, this comes at the cost of much more tedious calculations. For clarity and simplicity, we therefore 
confine ourselves to test functions $\phi_L(t)$ centered at the integers. 

\subsection*{Organization of the article} 
This article is organized as follows.  In Section \ref{FourierPrelim}, we use Fourier analysis to interpret Theorems \ref{main-elliptic} and \ref{main-modular} in terms of moments of appropriate trigonometric sums.  Sections \ref{Moments-ellcurv} and \ref{Moments-modform} contain the proofs of Theorems \ref{main-elliptic} and \ref{main-modular} respectively.  Sections \ref{Moments-ellcurv} and \ref{Moments-modform} are self contained and can be read independently of each other.

\subsection*{Acknowledgements} 
This work was initiated through discussions at IISER Pune.  The first and second named authors thank IISER Pune for its hospitality. The authors also thank the anonymous referee for suggestions that improved the exposition of this article.

\section{Preliminaries from Fourier analysis}\label{FourierPrelim}
In this section, we make a note about the different kinds of smooth functions considered in this article.
Let $\Phi \in C^{\infty}(\R)$ be an even, real-valued function.
  As in Theorems \ref{main-elliptic} and \ref{main-modular}, for $L \geq 1,$ let
$$\phi_L(t) := \sum_{m \in \Z}\Phi(L(t + m)).$$ 
Then, $\phi_L(t)$ is a periodic function with Fourier expansion
$$\phi_L(t) = \frac{1}{L}\sum_{m \in \Z}\widehat{\Phi}\left(\frac{m}{L}\right)e(mt).$$
Here, $e(x)$ denotes $e^{2\pi i x}.$
%Henceforth, in this article, depending on the context, any sum over primes $p \leq x$ will denote sums over primes $p$ which are coprime to the level $N$ of the newform $f \in \mathcal F_{N,k}$ or conductor $\mathcal N_E$ of the elliptic curve $E$ under consideration.
For a positive integer $N,$ we define
$$\pi_N(x) := \#\{p \leq x:\,p \text{ prime },\,p \nmid N\}.$$ 
Since $\widehat{\Phi}$ is an even function,
\begin{align}
N_{\Phi,L,f}(x) &:= \sum_{p \leq x \atop { p \nmid N}} \phi_L(\thtp) 
= \frac{1}{L}\sum_{m \in \Z}\widehat{\Phi}\left(\frac{m}{L}\right)\sum_{p \leq x \atop { p \nmid N}}e(m\thtp)\nonumber\\
&= \frac{1}{L}\widehat{\Phi}(0)\pi_N(x) + \frac{1}{L}\sum_{m \geq 1}\widehat{\Phi}\left(\frac{m}{L}\right)\sum_{p \leq x \atop { p \nmid N}}2 \cos 2\pi m\thtp.\nonumber
\end{align}
We now recall the following classical result (see for example, \cite[Lemma 1]{Serre}) that encodes recursive relations between $a_f(p^m),\,m \geq 1.$ 
\begin{lemma}%\label{mult}
For a prime $p$ and an integer $m \geq 1,$ 
\begin{equation*}\label{mult}
2 \cos 2 \pi m \thtp = a_f(p^{2m}) - a_f(p^{2m-2}).
\end{equation*}
\end{lemma}
Thus, we deduce
\begin{align*}\label{Nphi}
N_{\Phi,L,f}(x) &= \frac{1}{L}\widehat{\Phi}\left(0\right)\pi_N(x) + \frac{1}{L}\sum_{m \geq 1}\widehat{\Phi}\left(\frac{m}{L}\right)\sum_{p \leq x \atop { p \nmid N}} a_f(p^{2m}) - a_f(p^{2m-2})\nonumber\\
&=\frac{1}{L}\left(\widehat{\Phi}\left(0\right) - \widehat{\Phi}\left(\frac{1}{L}\right)\right)\pi_N(x) + \frac{1}{L}\sum_{m \geq 1} \left(\widehat{\Phi}\left(\frac{m}{L}\right) -  \widehat{\Phi}\left(\frac{m+1}{L}\right)\right) \sum_{p \leq x \atop{p \nmid N}}a_f(p^{2m}).
\end{align*}

Henceforth, for each $m,$ let us denote $\widehat{\phi_L}(m)$ as the $m$-th Fourier coefficient of the period 1 function $\phi_L(t).$  We have
$$\widehat{\phi_L}(m) = \frac{1}{L} \widehat{\Phi}\left(\frac{m}{L}\right).$$  Since $\widehat{\Phi}$ is even,
\begin{equation}\label{expected-value}
\frac{1}{L}\left(\widehat{\Phi}\left(0\right) - \widehat{\Phi}\left(\frac{1}{L}\right)\right) = \widehat{\phi_L}(0) -  \widehat{\phi_L}(1)  = \int_0^1 \phi_L(t)(1 - \cos 2\pi t) dt =  \int_0^1 \phi_L(t)\mu(t) dt.
\end{equation}
Define \begin{equation}\label{U(m)-defn}
U(m):= \frac{1}{L}\left(\widehat{\Phi}\left(\frac{m}{L}\right) -  \widehat{\Phi}\left(\frac{m+1}{L}\right)\right) \text{ for every } m \geq 1.
\end{equation}  We have
\begin{align}\label{subtracted mean}
N_{\Phi,L,f}(x) - \pi_N(x)  \int_0^1 \phi_L(t)\mu(t) dt = \sum_{m \geq 1} {U}(m)\sum_{  p \leq x \atop { p \nmid N}}a_f(p^{2m}).
\end{align}
If $x>6$, which we may assume without loss of generality, then following an identical calculation in the case of elliptic curves, we obtain
\begin{align}\label{subtracted mean-2}
N_{\Phi,L,E(a,b)}(x) - \tilde{\pi}(x)  \int_0^1 \phi_L(t)\mu(t) dt =  
\sum_{m \geq 1} {U}(m)\sum\limits_{\substack{x/2<p\le x \\ p\nmid \Delta(a,b)}} \tilde{a}_{E(a,b)}\left(p^{2m}\right),
\end{align}
where $\Delta(a,b)=4a^3+27b^2$ is the discriminant of $E(a,b)$. Here we note that the prime divisors exceeding 3 of the conductor and 
discriminant coincide. Since we consider primes $p$ in dyadic intervals $(x/2,x]$ and assume $x>6$, the primes $p=2,3$ are automatically 
excluded. 

\subsection{Integral representation of the variance}
As will be evident in the later sections, the second moment in both settings of elliptic curves and modular forms will be proportional to 
$$\sum_{m \geq 1} {U}(m)^2 = \frac{1}{L^2} \sum_{m \geq 1}\left(\widehat{\Phi}\left(\frac{m}{L}\right) - \widehat{\Phi}\left(\frac{m+1}{L}\right)\right)^2.$$ It is useful to know that this quantity has an integral representation that involves the measure $\mu$. We prove 
\begin{prop}\label{variance-exact}
	We have \begin{equation*}
\sum_{m\geq 1} {U}(m)^2 =\int_0^1\phi_L(t)^2\mu(t)dt - \left( \int_0^1 \phi_L(t) \mu(t) dt \right)^2.
	\end{equation*}
\end{prop}

\begin{proof}
Recall that for every $m \in \Z,$ the $m$-th Fourier coefficient of $\phi_L(t)$ satisfies
$$\widehat{\phi_L}(m) = \frac{1}{L}\widehat{\Phi}\left(\frac{m}{L}\right).$$
	We have 
	\begin{equation*}
	\begin{split}
	\sum_{m \geq 1}{U}(m)^2 &= \sum_{m\geq 1} (\widehat{\phi_L}(m)  - \widehat{\phi_L}(m+1)  )^2\\
	&= \sum_{m\geq 1} (\widehat{\phi_L}(m)^2  -2\p(m)\p(m+1) + \widehat{\phi_L}(m+1)^2).
	%(\overline{\widehat{\phi_L}(n)}  -\overline{ \widehat{\phi_L}(n+1)} )
	%	&= \sum_{n=1}^{\infty}( \p(n)^2 + \p(-n)^2 +\p(n+1)^2 +\p(-(n+1))^2\\
	%	& \quad + 2\sum_{n=1}^{\infty} (\p(n)\p(-n) + \p(n+1)\p(-(n+1)))\\
	%& \quad  -2\sum_{n=1}^{\infty} (\p(n)\p(n+1) + \p(-n)\p(-(n+1)) + \p(n)\p(-(n+1)) +\p(-n)\p(n+1)).
	\end{split}
	\end{equation*}
	Since $\widehat{\Phi}$ is an even function, we know that $\widehat{\phi_L}(m) = \widehat{\phi_L}(-m).$  Hence,
	\begin{align*}
	\sum_{m \in\Z} \p(m)^2 &= 2\sum_{m\geq 1}\p(m)^2 +\p(0)^2,\\
	\sum_{m \in\Z} \p(m+1)^2 &= 2\sum_{m\geq 1}\p(m+1)^2 +\p(0)^2 + 2\p(1)^2,\\
	2\sum_{m \in \Z} \p(m)\p(m+1) &= 4\sum_{m \geq 1}\p(m)\p(m+1) +4\p(0)\p(1).
	\end{align*} Therefore,
	\begin{align*}
	\sum_{m \geq 1}{U}(m)^2 &= \sum_{m \in\Z} \p(m)^2 - \sum_{m \in \Z} \p(m)\p(m+1) - \left( \p(0) -\p(1)\right)^2.
	\end{align*}
	For the first two terms, we use  Parseval's identity:  If  
	$$f(x)= \sum_{n \in \Z}\widehat{f}(n)e^{2\pi i n x}, \qquad g(x)=\sum_{n \in \Z}\widehat{g}(n)e^{2\pi i n x},$$ 
	then \begin{equation*}
	\sum_{n\in \Z}\widehat{f}(n)\overline{\widehat{g}(n)} = \int_0^1f(t)\overline{g(t)}dt.
	\end{equation*}
	Using Parseval's identity and equation \eqref{expected-value},
	we now obtain 
	\begin{align*}
	\sum_{m \geq 1} {U}(m)^2 &= \int_0^1 \phi_L(t)^2dt  - \int_0^1 \phi_L(t)\phi_L(t)\cos(2\pi t)dt -  \left( \int_0^1 \phi_L(t) (1-\cos(2\pi t))dt \right)^2 \\
	& = \int_0^1\phi_L(t)^2\mu(t)dt - \left( \int_0^1 \phi_L(t) \mu(t) dt \right)^2.
	\end{align*}
	
\end{proof}
{\bf Note:} If we do not assume that $\phi_L$ is even, then we get $$\sum_{m \geq 1}{U}(m)^2 =  \frac{1}{2}\int_0^1\phi_L(t)^2\mu(t)dt + \frac{1}{2}\int_0^1\phi_L(t)\phi_L(-t)\mu(t)dt - \left( \int_0^1 \phi_L(t) \mu(t) dt \right)^2.$$

%\subsection{Remarks on proofs of Theorems \ref{main-elliptic} and \ref{main-modular}}\label{proof-sketch} 
We will see later that our main theorems will be proved by showing that for each $r\in \N$,  
$$ \frac{1}{4AB} \sum\limits_{|a|\le A} \sum\limits_{|b|\le B}  \left(\sum_{m \geq 1} {U}(m)\sum\limits_{\substack{x/2<p\le x \\ p\nmid \Delta(a,b)}} \tilde{a}_{E(a,b)}\left(p^{2m}\right)\right)^r$$
%$$\frac{1}{|\mathcal F_{N,k}|}\sum_{f \in \mathcal F_{N,k}}\left(N_{\Phi,L,f}(x) -\pi(x)\int_0^1\phi_L(t)\mu(t)dt\right)^r =$$
and
$$\frac{1}{|\mathcal F_{N,k}|}\sum_{f \in \mathcal F_{N,k}} \left( \sum_{m \geq 1}{U}(m)\sum_{  p \leq x \atop { p \nmid N}}a_f(p^{2m})\right)^r$$ 
%$$\frac{1}{4AB} \sum\limits_{|a|\le A} \sum\limits_{|b|\le B} \left(N_{\Phi,L,E(a,b)}(x) - \tilde{\pi}(x)  \int_0^1 \phi_L(t)\mu(t) dt \right)^r $$
 converge to the moments of the Gaussian distribution under the respective hypotheses in each case.

\section{Computation of moments: elliptic curves} \label{Moments-ellcurv}
In this section, we prove Theorem \ref{main-elliptic} closely proceeding along the lines in \cite{BP19}. 
We shall write ``MRH'' for ``Modular Riemann Hypothesis'', by 
which we mean that the Riemann Hypothesis holds for all $L$-functions associated to Hecke eigenforms with respect to the full modular group. 
The truth of MRH is 
assumed in Theorem \ref{main-elliptic}. All
$\O$-constants in this section may depend on the parameters $\lambda$ and $\omega$.  
Throughout this section, we assume that $x>6$, and the parameter $\varepsilon>0$ is arbitrarily small and may change from line to line.  

Let $A,B,L\ge 1$. In Theorem \ref{main-elliptic}, we fix $L$. For the time being, though, until we reach Corollary \ref{Etappfin} below,
we treat $L$ as a variable as well. As mentioned in the previous section, our goal is to compute the moments
\begin{equation} \label{Erdef}
E_r:=\frac{1}{4AB} \sum\limits_{0<|a|\le A} \sum\limits_{0<|b|\le B} 
\left( N_{\Phi,L,E(a,b)}(x)- \tilde{\pi}(x) \int\limits_{0}^1 \phi_L(t) \mu(t) \ dt\right)^r.
\end{equation}
More precisely, here we want to isolate a main term in our evaluation of $E_r$ under MRH, which we have not been able to do in \cite{BP19} when
we considered the characteristic function of an interval in place of $\phi_L$. (In \cite{BP19}, we achieved this only under some 
stronger hypotheses.) Only if we are 
able to isolate a main term can a Central Limit Theorem for the error in the Sato-Tate law be established. We first prove the following
Lemmas which will be needed later.

\begin{lemma} \label{Umbound}
We have 
$$
U(m)= \O\left(\min\left\{L^{-2},L^{-1}e^{-\lambda(m/L)^{\omega}}\right\}\right),
$$
where the implied constant depends only on $\Phi$.
\end{lemma}

\begin{proof}
The upper bound by the second term in the minimum is a direct consequence of the definition of $U(m)$ in \eqref{U(m)-defn} and the bound  
$\Phi(t)\ll e^{-\lambda|t|^{\omega}}$. Moreover, 
using \eqref{U(m)-defn} and the mean value theorem from calculus, we find that
$$
|U(m)| \le \frac{1}{L^2} \cdot \sup\limits_{t\in \mathbb{R}} |\widehat{\Phi}'(t)|,
$$
which completes the proof.
\end{proof}

\begin{lemma} \label{Umsum}
Set 
\begin{equation} \label{Mdef}
M:= \lceil L(\log x)^{1/\omega+\varepsilon} \rceil
\end{equation}
and assume that $L\le x^{672}$. Then we have 
\begin{equation} \label{square}
\sum\limits_{m\ge 1} U(m)^2 \asymp L^{-3} \asymp \sum\limits_{1\le m\le M} U(m)^2  \quad \mbox{and} \quad \sum\limits_{m\ge 1} U(m)^2=
\sum\limits_{1\le m\le M} U(m)^2 +O\left(x^{-2019}\right) 
\end{equation}
and
\begin{equation} \label{modulus} 
\sum\limits_{m\ge M} |U(m)| = O\left(x^{-2019}\right).
\end{equation}
\end{lemma}

\begin{proof} We begin with proving that
\begin{equation} \label{begin}
\sum\limits_{m\ge 1} U(m)^2 \asymp L^{-3}.
\end{equation}
Recall Proposition \ref{variance-exact}, in which the sum of $U(m)^2$ was expressed in terms of integrals involving $\phi_L(t)$.
From 
$$
\phi_L(t) = \sum_{m \in \Z} \Phi(L(t+m)) \quad \mbox{and} \quad \Phi(t) \ll (1+|t|)^{-2},
$$
we deduce that 
$$
\phi_L(t)=\Phi(L(t-1))+\Phi(Lt)+ \O\left(\sum\limits_{m\ge 1} (1+Lm)^{-2}\right)
$$
if $0\le t\le 1$. Since
$$
\sum\limits_{m\ge 1} (1+mL)^{-2} \ll \int\limits_1^{\infty} (1+Lx)^{-2} dx\ll L^{-2},
$$
it follows that 
\begin{equation*}
\begin{split}
\int\limits_{0}^{1} \Phi(Lt)^2\mu(t)dt + \O\left(L^{-4}\right) \le & \int_0^1\phi_L(t)^2\mu(t)dt \\ 
\le & 
2\int\limits_{0}^{1} \left(\Phi(L(t-1))^2+\Phi(Lt)^2\right)\mu(t)dt +\O\left(L^{-4}\right)\\
= & 4\int\limits_{0}^{1} \Phi(Lt)^2\mu(t)dt + \O\left(L^{-4}\right),
\end{split}
\end{equation*}
where we recall that $\Phi$ is even. 
Furthermore, 
$$
\int\limits_{0}^{1} \Phi(Lt)^2\mu(t)dt = \frac{1}{L} \int\limits_0^L \Phi(y) \cdot 2\sin^2 \frac{\pi y}{L} dy \sim \frac{2\pi^2}{L^3} 
\int\limits_{0}^{\infty} \Phi(y)y^2dy
$$
as $L\rightarrow\infty$. 
We deduce that
$$
\int_0^1\phi_L(t)^2\mu(t)dt \asymp L^{-3}.
$$
Similarly, we find
$$
\left(\int_0^1\phi_L(t)\mu(t)dt\right)^2 = \O\left(L^{-4}\right).
$$
Now \eqref{begin} follows from Proposition \ref{variance-exact}.

Next we turn to proving \eqref{modulus}. Using Lemma \ref{Umbound} together with Bernoulli's inequality, we have
\begin{equation*}
\begin{split}
\sum\limits_{m\ge M} |U(m)| \ll & L^{-1}\sum\limits_{m\ge M} e^{-\lambda(m/L)^{\omega}}\\ = &
L^{-1}\sum\limits_{m\ge 0} e^{-\lambda (M/L)^{\omega}(1+m/M)^{\omega}}\\ \le &
L^{-1}\sum\limits_{m\ge 0} e^{-\lambda (M/L)^{\omega}(1+m\omega/M)}\\
= & L^{-1}e^{-\lambda(M/L)^{\omega}}\left(1-e^{-\lambda(M/L)^{\omega}\omega/M}\right)^{-1}\\
\le & L^{-1}e^{-\lambda(M/L)^{\omega}}\cdot \frac{M}{\lambda(M/L)^{\omega}\omega}.
\end{split}
\end{equation*}
By definition of $M$ in \eqref{Mdef}, it follows that
$$
\sum\limits_{m\ge M} |U(m)| \ll  (\lambda\omega)^{-1}(\log x)^{(1/\omega+\varepsilon)(1-\omega)}x^{-\lambda(\log x)^{\omega\varepsilon}} \ll
x^{-2019}
$$
and hence, \eqref{modulus} is established. 

It remains to show that
$$
\sum\limits_{m\ge 1} U(m)^2=\sum\limits_{1\le m\le M} U(m)^2 +O\left(x^{-2019}\right).
$$
From this, \eqref{begin} and our condition $L\le x^{672}$, it then follows that
$$
\sum\limits_{1\le m\le M} U(m)^2 \asymp L^{-3}
$$
and hence \eqref{square} is established. Indeed, using Lemma \ref{Umbound} and \eqref{modulus}, we have
$$
\sum\limits_{m>M} U(m)^2\ll \sum\limits_{m> M} |U(m)|\ll x^{-2019}, 
$$
which completes the proof.
\end{proof}

To be consistent with \cite{BP19}, we write \eqref{subtracted mean-2} in the form 
\begin{equation*}
N_{\Phi,L,E(a,b)}(x)- \tilde{\pi}(x) \int\limits_{0}^1 \phi_L(t) \mu(t) \ dt = 
\sum\limits_{m\ge 1} \tilde{U}(m) \sum\limits_{\substack{x/2<p\le x \\ p\nmid \Delta(a,b)}}
\tilde{a}_{E(a,b)}\left(p^{m}\right), 
\end{equation*}
where 
$$
\tilde{U}(m):=\begin{cases} U(m/2) & \mbox{if } m \mbox{ is even,}\\ 0 & \mbox{otherwise.} \end{cases}
$$
We define $M$ as in \eqref{Mdef} and assume $L\le x^{672}$ throughout this section.
Using \eqref{modulus} and $\tilde{a}_{E(a,b)}\left(p^m\right)= \O(1)$, we may cutoff the sum over $m$ at $2M$, at the cost of a 
negligible error, obtaining
\begin{equation*} \label{neg}
\begin{split}
N_{\Phi,L,E(a,b)}(x)- \tilde{\pi}(x) \int\limits_{0}^1 \phi_L(t) \mu(t) \ dt = &
\sum\limits_{1\le m\le 2M} \tilde{U}(m) \sum\limits_{\substack{x/2<p\le x \\ p\nmid \Delta(a,b)}}
\tilde{a}_{E(a,b)}\left(p^m\right) +  \O_{\varepsilon}\left(x^{-2018}\right)\\
= & \O_{\varepsilon}\left(x^{1+\varepsilon}L^{-1}\right),
\end{split}
\end{equation*}
where the second line arises from $\tilde{U}(m)\ll 1/L^2$ (by Lemma \ref{Umbound}) and the definition of $M$ in \eqref{Mdef}. 
Further,
as in \cite[section12]{BP19}, we separate the contribution of primes $p$ dividing $ab\not=0$. Using $\tilde{a}_{E(a,b)}\left(p^m\right)= \O(1)$,
$\omega(ab)=\O(\log(2ab))$ and again $\tilde{U}(m)\ll 1/L^2$ and \eqref{Mdef}, this can be done at the cost of a small error as well, namely, we obtain 
\begin{equation*} \label{rempab}
\begin{split}
\sum\limits_{1\le m\le 2M} \tilde{U}(m)  \sum\limits_{\substack{x/2<p\le x \\ p\nmid \Delta(a,b)}} 
\tilde{a}_{E(a,b)}\left(p^m\right) = & \sum\limits_{1\le m\le 2M} \tilde{U}(m) \sum\limits_{\substack{x/2<p\le x \\ p\nmid ab\Delta(a,b)}}
\tilde{a}_{E(a,b)}\left(p^m\right)\\ & + \O\left(\log(2|ab|)x^{\varepsilon}L^{-1}\right). 
\end{split}
\end{equation*}
Set 
\begin{equation*} 
\mathcal{L}:=\log(2AB).
\end{equation*}
Recall the definition of $E_r$ in \eqref{Erdef}. It follows that
\begin{equation*}
\begin{split}
E_r= & \frac{1}{4AB}  \sum\limits_{|a|\le A} \sum\limits_{|b|\le B} \Bigg(\sum\limits_{1\le m\le 2M} \tilde{U}(m) 
\sum\limits_{\substack{x/2<p\le x\\
		p\nmid ab\Delta(a,b)}}  \tilde{a}_{E(a,b)}(p^m) 
		+\O\left(\mathcal{L}x^{\varepsilon}L^{-1}\right)\Bigg)^r,
\end{split}
\end{equation*}
where we note that the inner-most sum over $p$ on the right-hand side is empty if $a=0$ or $b=0$. Now we define
$$
F_r:=\frac{1}{4AB}  \sum\limits_{|a|\le A} \sum\limits_{|b|\le B} \Bigg(\sum\limits_{1\le m\le 2M} \tilde{U}(m) 
\sum\limits_{\substack{x/2<p\le x\\
		p\nmid ab\Delta(a,b)}}  \tilde{a}_{E(a,b)}(p^m) \Bigg)^r
$$
and
$$
\tilde{F}_r:=\frac{1}{4AB}  \sum\limits_{|a|\le A} \sum\limits_{|b|\le B} \left| \sum\limits_{1\le m\le 2M} 
\tilde{U}(m) \sum\limits_{\substack{x/2<p\le x\\
		p\nmid ab\Delta(a,b)}}  \tilde{a}_{E(a,b)}(p^m) \right|^r.
$$
Then using the binomial formula, we deduce that
\begin{equation*} 
E_r=F_r+ \O_{\varepsilon,r}\left(\sum\limits_{s=0}^{r-1} \tilde{F}_s\left(\mathcal{L}x^{\varepsilon}L^{-1}\right)^{r-s}\right).
\end{equation*}
Using the Cauchy-Schwarz inequality, we observe that
\begin{equation*} \label{in}
\tilde{F}_s\le \tilde{F}_{2s}^{1/2}= F_{2s}^{1/2},
\end{equation*}
where the last equation holds because the coefficients $\tilde{U}(m)$ are real. Hence, we have
\begin{equation} \label{EtXt}
E_r=F_r+ \O_{\varepsilon,r}\left(\sum\limits_{s=0}^{r-1} F_{2s}^{1/2}\left(\mathcal{L}x^{\varepsilon}L^{-1}\right)^{r-s}\right).
\end{equation}

Now we continue in a similar way as in \cite{BP19}. However, there are two significant differences between the setting in the present paper and that in \cite{BP19}: Firstly, in \cite{BP19}, 
the coefficients $U(m)$ satisfied the bound $U(m)\ll 1/m$, whereas here we have the bound $U(m)\ll 1/L^2$. 
Secondly, in \cite{BP19}, we derived our 
final moment bound under the condition that $M\ge \tilde{\pi}(x)^{1/2}$, which is not the case here. In the following, we perform the required
adjustments to make our method in \cite{BP19} work in the context of this paper. 

Opening the $r$-power and applying identities for the coefficients, we established in \cite{BP19} (with $M$ in place of $2M$) that 
\begin{equation} \label{fourth}
\begin{split}
F_r
= &  \sum\limits_{u=1}^r 
\sum\limits_{\alpha_1=0}^{\infty} \cdots \sum\limits_{\alpha_u=0}^{\infty} 
C(\alpha_1,\ldots,\alpha_u) \sum\limits_{\substack{x/2< p_1,\ldots,p_u\le x\\  p_{\rho}\not = p_{\sigma} \mbox{\scriptsize\ if } 1\le \rho<\sigma\le u}}
\frac{1}{4AB}\times\\ & \mathop{\sum\limits_{|a|\le A} \sum\limits_{|b|\le B}}_{(ab\Delta(a,b),p_1\ldots p_u)=1} 
\tilde{a}_{E(a,b)}\left(p_1^{\alpha_1}\cdots p_u^{\alpha_u}\right)
\end{split}
\end{equation}
with
\begin{equation*}
C(\alpha_1,\ldots,\alpha_u):=\sum\limits_{\{1,\ldots,r\}=\mathcal{S}_1\dot\cup\cdots \dot\cup \mathcal{S}_u} 
\sum\limits_{1\le m_1,\ldots,m_r\le 2M} \tilde{U}(m_1)\cdots \tilde{U}(m_r) \prod\limits_{j=1}^u 
D\left(\left(m_i\right)_{i\in \mathcal{S}_j};\alpha_j\right),
\end{equation*}
where the numbers $D(m_1,\ldots,m_r,m)$ are integers for which the following was proved in \cite{PrabhuSinha}.

\begin{lemma} \label{prodpolys} Assume that $m_1,...,m_r\in \mathbb{N}$ and $E$ has good reduction at $p$. 
	Set $\Sigma=m_1+\cdots+m_r$. Then
	\begin{equation*} \label{products}
	\prod\limits_{i=1}^r \tilde{a}_E\left(p^{m_i}\right) = \sum\limits_{m=0}^{\infty} D(m_1,\ldots,m_r;m)\tilde{a}_E(p^m),
	\end{equation*}
	where $D(m_1,\ldots,m_r;m)$ are nonnegative integers satisfying
	\begin{equation*} \label{Dbounds}
	\begin{split}
	D(m_1,\ldots,m_r,m)= & 0 \quad \mbox{ if } m>\Sigma, \\
	D(m_1,\ldots,m_r;m)= & \O\left(\Sigma^{r-2}\right) \quad \mbox{ if } r\ge 2 \mbox{ and } 1\le m\le \Sigma,\\
	D(m_1,\ldots,m_r;0)= & \O\left(\Sigma^{r-3}\right) \quad \mbox{ if } r\ge 3,\\
	D(m_1,m_2;0) = & \begin{cases} 1 & \mbox{ if } m_1=m_2\\ 0 & \mbox{ if } m_1\not=m_2, \end{cases}\\
	D(m_1;m)= & \begin{cases} 1 & \mbox{ if } m_1=m\\ 0 & \mbox{ if } m_1\not=m.\end{cases} 
	\end{split}
	\end{equation*}
\end{lemma}

Using Lemma \ref{prodpolys} and $\tilde{U}(m)\ll 1/L^2$, we deduce the following results on $C(\alpha_1,\ldots,\alpha_r)$.

\begin{lemma} \label{Cbounds}
	Set 
	\begin{equation} \label{zn}
	z:=\sharp\{i\in \{1,\ldots,u\} :\alpha_i=0\} \quad \mbox{and} \quad n:=\sharp\{i\in \{1,\ldots,u\} : \alpha_i\not=0\}.
	\end{equation}
	Then 
	\begin{eqnarray}
	\label{Calphaest}
	C(\alpha_1,\ldots,\alpha_u) & = & \O_r\left(M^{2r-2n-3z}L^{-2r}\right) \quad \mbox{ if } 2z+n\le r\\
	\label{toomuch1}
	C(\alpha_1,\ldots,\alpha_u) & = & 0 \quad \mbox{ if } 2z+n>r \\    
	\label{000ext}
	C(0,\ldots,0)& = & \frac{(2z)!}{2^zz!}\cdot Z^z\quad \mbox{ if } 2z=r \\  
	\label{toomuch}
	C(\alpha_1,\ldots,\alpha_u) & = & 0 \quad \mbox{ if } \alpha_i>2rM \mbox{ for an } i\in \{1,\ldots,u\},  
	\end{eqnarray}
	where 
	\begin{equation} \label{Zdef}
	Z:=\sum\limits_{1\le m\le 2M} \tilde{U}(m)^2=\sum\limits_{1\le m\le M} U(m)^2.
	\end{equation}
\end{lemma}

In particular, using \eqref{Mdef} and \eqref{square}, we deduce the general bound
\begin{equation} \label{genCbound}
 C(\alpha_1,\ldots,\alpha_u)=\O_{\varepsilon,r,u}\left(x^{\varepsilon}M^{-2u}\right). 
\end{equation}

In \cite[subsection 10.4.]{BP19}, we proved the asymptotic estimate
\begin{equation} \label{redtoS}
\begin{split}
& \frac{1}{4AB}  
\mathop{\sum\limits_{|a|\le A} \sum\limits_{|b|\le B}}_{(ab\Delta(a,b),p_1\cdots p_u)=1} \tilde{a}_{E(a,b)}\left(p_1^{\alpha_1}\cdots p_u^{\alpha_u}\right)\\
= & S\left(p_1^{\alpha_1}\right)\cdots S\left(p_u^{\alpha_u}\right) + \O_{u,\varepsilon}\left(
\prod\limits_{i=1}^u \left(\alpha_i+1\right)\cdot x^{u/2+\varepsilon}\left(A^{-1}+B^{-1}\right)\right)
\end{split}
\end{equation}
for the averages of coefficients in \eqref{fourth}, where the function $S(n)$ is defined as 
\begin{equation*} \label{Sn}
S(n):=\frac{1}{s(n)^2} \mathop{\sum\limits_{a=1}^{s(n)} \sum\limits_{b=1}^{s(n)}}_{(ab\Delta(a,b),n)=1} \tilde{a}_{E(a,b)}(n),
\end{equation*}
$s(n)$ being the largest square-free integer dividing $n$.
Combining \eqref{fourth}, \eqref{toomuch}, \eqref{genCbound} and \eqref{redtoS}, we obtain
\begin{equation} \label{fifth}
\begin{split}
F_r = &  \sum\limits_{u=1}^r 
\sum\limits_{\alpha_1=0}^{2rM} \cdots \sum\limits_{\alpha_u=0}^{2rM} 
C(\alpha_1,\ldots,\alpha_u) \sum\limits_{\substack{x/2< p_1,...,p_u\le x\\ p_{\rho}\not = p_{\sigma} \mbox{\scriptsize\ if } 1\le \rho<\sigma\le u}}
S\left(p_1^{\alpha_1}\right)\cdots S\left(p_u^{\alpha_u}\right)\\ & 
+ \O_{\varepsilon,r}\left(x^{3r/2+\varepsilon}\left(A^{-1}+B^{-1}\right)\right).
\end{split}
\end{equation}
To bound the sums over primes in \eqref{fifth}, we established the following in \cite[subsection 10.6.]{BP19} (with $rM$ in place of $2rM$). 

\begin{lemma} \label{Save} Assume that $1\le \alpha_i\le 2rM$ for $i\in \{1,\ldots,u\}$. Then, under MRH, we have
	\begin{equation*} 
	\sum\limits_{\substack{x/2< p_1,\ldots,p_u\le x\\ p_{\rho}\not = p_{\sigma} \mbox{\scriptsize\ if } 
	1\le \rho<\sigma\le u}} S\left(p_1^{\alpha_1}\right)\cdots
	S\left(p_u^{\alpha_u}\right)
	\ll \alpha_1\cdots \alpha_u (\log x)^u.
	\end{equation*}
\end{lemma}

We also want to include the $\alpha_i$'s with $\alpha_i=0$ in our result. Taking into account that $S(1)=1$, we deduce the following 
immediately from Lemma \ref{Save}.

\begin{lemma} \label{Saves} Assume that $0\le \alpha_i\le 2rM$ for $i\in \{1,\ldots,u\}$. Then, under MRH, we have
	\begin{equation*}  
	\sum\limits_{\substack{x/2< p_1,\ldots,p_u\le x\\ p_{\rho}\not = p_{\sigma} \mbox{\scriptsize\ if } 1\le \rho<\sigma\le u}} S\left(p_1^{\alpha_1}\right)\cdots
	S\left(p_u^{\alpha_u}\right)
	\ll \tilde{\pi}(x)^z(rM)^n(\log x)^n 
	\end{equation*}
	with $z$ and $n$ as defined in \eqref{zn}.
\end{lemma}

Combining \eqref{fifth} with Lemmas \ref{Cbounds} and \ref{Saves}, we get the following by a short calculation using $u=n+z$.

\begin{prop} \label{Xtapp} Fix $r\in \mathbb{N}$ and $\varepsilon>0$ and assume that $M\le \tilde{\pi}(x)^{1/2}$. Then, under MRH, we have
	\begin{equation} \label{Xtfinal}
	\begin{split}
	F_r= & \delta(r)\cdot \frac{r!}{2^{r/2}(r/2)!} \cdot \left(\tilde{\pi}(x)\sum\limits_{1\le m\le M} U(m)^2\right)^{r/2} + 
	\O_{\varepsilon,r}\left(x^{3r/2+\varepsilon}\left(A^{-1}+B^{-1}\right)\right)+
	\\ & 
	\O_{\varepsilon,r}\left(\left(\tilde{\pi}(x)M^{-2}\right)^{[(r-1)/2]}x^{\varepsilon}\right), 
	\end{split}
	\end{equation}
	where
	$$
	\delta(r)= \begin{cases} 1 & \mbox{ if } r \mbox{ is even,}\\ 0 & \mbox{ if } r \mbox{ is odd.} \end{cases}
	$$
\end{prop}

We note that the main term on the right-hand side of \eqref{Xtfinal} comes from the contribution of $u=r/2$ and $\alpha_1=\cdots=\alpha_u=0$ 
(and hence, $z=r/2$ and $n=0$) to the 
right-hand side of \eqref{fifth} if $r$ is even. In all other cases, by \eqref{toomuch1}, we have necessarily $z\le [(r-1)/2]$ if 
$C(\alpha_1,\ldots,\alpha_u)\not=0$, which is the reason for the exponent $[(r-1)/2]$ in the error term. 
Using \eqref{Mdef}, \eqref{square} and the trivial inequality $[(r-1)/2]\le (r-1)/2$, the following is a consequence of Proposition \ref{Xtapp}. 

\begin{cor} \label{Xtcor} Fix $r\in \mathbb{N}$ and $\varepsilon>0$. Then, under MRH, we have
	\begin{equation*}
	F_r= \left(\tilde{\pi}(x)\sum\limits_{m\ge 1} U(m)^2\right)^{r/2} \cdot 
	\left(\delta(r)\cdot \frac{r!}{2^{r/2}(r/2)!} + \O_{\varepsilon,r}\left((\log x)^{-1}\right)\right), 
	\end{equation*}
	provided that
	\begin{equation*}
	L\le x^{1/(r+2)-\varepsilon}
	\end{equation*}
	and 
	\begin{equation*}
	A,B\ge \left(L^{3/2}x\right)^{r}x^{\varepsilon}.
	\end{equation*}
\end{cor}

Using \eqref{square}, \eqref{EtXt} and Corollary \ref{Xtcor}, we deduce the following for the moment $E_r$ in question. 

\begin{cor} \label{Etappfin} Fix $r\in \mathbb{N}$ and $\varepsilon>0$. Then, under MRH, we have
	\begin{equation*} 
	E_r = \left(\tilde{\pi}(x)\sum\limits_{m\ge 1} U(m)^2\right)^{r/2}
	\left(\delta(r)\cdot \frac{r!}{2^{r/2}(r/2)!} + \O_{\varepsilon,r}\left((\log x)^{-1}\right)\right), 
	\end{equation*}
	provided that
	\begin{equation*}
	L\le x^{1/\max\{2r,r+2\}-\varepsilon},
	\end{equation*}
	\begin{equation*} 
	A,B\ge \left(L^{3/2}x\right)^{\max\{2(r-1),r\}}x^{\varepsilon}
	\end{equation*}
	and 
	\begin{equation*}
	 \log(2AB)\le x^{1/2-\varepsilon}.
	\end{equation*}
\end{cor}

This together with Proposition \ref{variance-exact} implies the result in Theorem \ref{main-elliptic}, where we had fixed the number $L$.

\section{Computation of moments: Modular forms}\label{Moments-modform}
In this section, we prove Theorem \ref{main-modular}.  Before we proceed with the proof, we record the following proposition. 
\begin{prop}\label{cut-off}
Suppose $\Phi \in C^{\infty}(\R)$ is a real-valued, even function. For $f \in \mathcal F_{N,k},$ as defined before, 
    $$N_{\Phi,L,f}(x) := \sum_{p \leq x \atop{p \nmid N}}\phi_L(\thtp).$$
\begin{enumerate}
\item[{\bf(a)}] Suppose the Fourier transform $\widehat{\Phi}$ is compactly supported in the interval $[-B,B].$  Then,
\begin{align*}\label{cut-off-0}
N_{\Phi,L,f}(x) -\pi_N(x)\int_0^1\phi_L(t)\mu(t)dt =  \sum\limits_{1\leq m\leq M} {U}(m)\sum_{  p \leq x \atop{p \nmid N}}a_f(p^{2m}),
\end{align*}
with $M \asymp BL.$

\item[{\bf(b)}] Suppose $\widehat{\Phi}(t) \ll e^{-\lambda |t|^{\omega}}$ for some $\lambda,\omega >0$ as $|t| \to \infty.$  Then,
\begin{align*}
N_{\Phi,L,f}(x) -\pi_N(x)\int_0^1\phi_L(t)\mu(t)dt =  \sum\limits_{1\leq m\leq M} {U}(m)\sum_{  p \leq x \atop{p \nmid N}}a_f(p^{2m}) + \O\left(x^{-2019}\right),
\end{align*}
for
$$M \asymp L\left(\frac{2020 \lambda}{1 - \epsilon} \log x\right)^{\frac{1}{\omega}}.$$
\end{enumerate}
For the above choice of $M$ in (b), we also have
\begin{equation}\label{squares}
\sum_{m=1}^M U(m)^2 = \sum_{m \geq 1} U(m)^2 +  \O(x^{-2019}).
\end{equation}

\end{prop}
\begin{proof}
By equation \eqref{subtracted mean}, we have,
\begin{align*}
N_{\Phi,L,f}(x) - \pi_N(x)  \int_0^1 \phi_L(t)\mu(t) dt = \sum_{m \geq 1} {U}(m)\sum_{  p \leq x \atop {p \nmid N}}a_f(p^{2m}).
\end{align*}
Part (a) is immediate since $U(m) = 0$ for $m > BL.$
To prove part (b), we use the Ramanujan-Deligne bound,
$$|a_f(p^{2m})| \leq 2m+1.$$
Thus, for any positive integer $M,$
$$\sum_{m > M} {U}(m)\sum_{  p \leq x \atop {p \nmid N}}a_f(p^{2m}) \ll \pi(x) \sum_{m >M} m|{U}(m)|.$$
In order to prove part (b), we observe that
\begin{equation*}
\begin{split}
\pi(x)\sum_{m>M}\frac{m/L}{e^{\lambda(m/L)^{\omega}}} &\ll \pi(x) \int_M^{\infty} \frac{y/L}{e^{\lambda(y/L)^{\omega}}}dy\\
&\ll \pi(x)L\frac{1}{\omega \lambda^{1/\omega}}\int_{e^{\lambda(M/L)^{\omega}}}^{\infty}\frac{(\log t)^{\frac{2}{\omega} - 1}}{t^2} dt\\
&\ll  \pi(x)L\frac{1}{\omega \lambda^{1/\omega}}\left({e^{\lambda(M/L)^{\omega}}}\right)^{-1 + \epsilon}.
\end{split}
\end{equation*}
Choosing $$M \asymp L\left(\frac{2020 \lambda}{1 - \epsilon} \log x\right)^{\frac{1}{\omega}},$$
we get
$$\pi(x)L\frac{1}{\omega \lambda^{1/\omega}}\left({e^{\lambda(M/L)^{\omega}}}\right)^{-1 + \epsilon} \ll x^{-2019}.$$
Equation \eqref{squares} can be obtained in a similar manner.
\end{proof}

After describing some preliminary tools in Sections \ref{identities} and \ref{Chebyshev}, we prove Theorem \ref{main-modular} in Section \ref{modular-proof}.
\subsection{Trace formula and estimates}\label{identities}
We will use the following trace formula repeatedly while computing the required moments for Theorem \ref{main-modular}. 
\begin{prop}\label{EStrace}
	Let $n$ be a positive integer coprime to $N.$  Then, 
	\begin{equation}\label{trace-1}
 \sum_{f \in \mathcal F_{N,k}} a_f(n) = \left(|\mathcal F_{N,k}|  + \O\left(\sqrt{N}\right)\right)\left(\begin{cases}
\frac{1}{\sqrt n}&\text{ if }n \text{ is a square},\\
0&\text{ otherwise}
\end{cases}\right)
+ \O\left(n\sigma_0(n)4^{\nu(N)}\right),
\end{equation}
where $\sigma_0(n)$ denotes the number of divisors of $n$ and $\nu(N)$ denotes the number of distinct prime divisors of $N.$
Thus,
\begin{equation}\label{trace-2}
\frac{1}{|\mathcal F_{N,k}|}\sum_{f \in \mathcal F_{N,k}} a_f(n) = \left(\begin{cases}
\frac{1}{\sqrt n}&\text{ if }n \text{ is a square},\\
0&\text{ otherwise}
\end{cases} \right)+ \O\left(n\sigma_0(n)\frac{4^{\nu(N)}}{k\sqrt N}\right).
\end{equation}

\end{prop}

	%for any integer $m \geq 1,$
	%$$\sum_{f \in \mathcal F_{N,k}} a_f(p^{2m}) = \frac{|\mathcal F_{N,k}|}{p^m} + \O\left(\frac{\sqrt{N}}{p^m} + p^{4m+1} 4^{\nu(N)}\right).$$
\begin{proof}
	In order to prove this proposition, we draw upon the trace formula for Hecke operators acting on the subspace of primitive cusp forms in $S(N,k)$  and estimates for the terms of this trace formula.  These have been worked out in \cite{MS2}.
	
	By Remark 11 of \cite{MS2}, we have,
	\begin{equation}\label{dim}
	|\mathcal F_{N,k}| = NB_1(N)\frac{k-1}{12} + \O(\sqrt{N}),
	\end{equation}
where $B_1(N)$ is a multiplicative function such that for a prime power $q^r,$
$$B_1(q^r) = \begin{cases}
1-\frac{1}{q}&\text{ if } r=1,\\
1-\frac{1}{q}-\frac{1}{q^2}&\text{ if } r=2,\\
\left(1-\frac{1}{q}\right)\left(1-\frac{1}{q^2}\right)&\text{ if }r\geq 3.
\end{cases}$$
Moreover, following the trace formula and estimation of its terms in Section 3 of \cite{MS2} (see proof of \cite[Proposition 14]{MS2}), we have, for $n > 1$ and $(n,N) = 1$,
\begin{equation}\label{trace-new}
\sum_{f \in \mathcal F_{N,k}} a_f(n) = NB_1(N)\frac{k-1}{12}\left(\begin{cases}
\frac{1}{\sqrt n}&\text{ if }n \text{ is a square},\\
0&\text{ otherwise}
\end{cases}\right)
 + \O\left(n\sigma_0(n)4^{\nu(N)}\right).
\end{equation}
Combining equations \eqref{dim} and \eqref{trace-new}, we derive equation \eqref{trace-1}.
Now, we observe that $|\mathcal F_{N,k}| \asymp c NkB_1(N),$ for an absolute constant $c >0.$  By the formula for $B_1(N),$  we observe that
$$|\mathcal F_{N,k}| \asymp c NkB_1(N) \geq c \frac{Nk}{4^{\nu(N)}}.$$
Thus,
$$\frac{\sqrt{N} }{|\mathcal F_{N,k}|} \ll \frac{4^{\nu(N)}}{k\sqrt N}.$$
Hence,
\begin{equation*}
\begin{split} 
&\frac{1}{|\mathcal F_{N,k}|}\sum_{f \in \mathcal F_{N,k}} a_f(n) \\
&=\frac{1}{|\mathcal F_{N,k}|}\left(|\mathcal F_{N,k}|\begin{cases}
\frac{1}{\sqrt n}&\text{ if }n \text{ is a square},\\
0&\text{ otherwise}
\end{cases}\right) + \O\left(\frac{\sqrt{N}}{\sqrt{n}}\right) + \O\left(n\sigma_0(n) \frac{4^{\nu(N)}}{|\mathcal F_{N,k}|}\right)\\
&= \left(\begin{cases}
\frac{1}{\sqrt n}&\text{ if }n \text{ is a square},\\
0&\text{ otherwise}
\end{cases}\right) + \O\left(n\sigma_0(n)\frac{\sqrt{N} }{|\mathcal F_{N,k}|}\right)\\
&= \left(\begin{cases}
\frac{1}{\sqrt n}&\text{ if }n \text{ is a square},\\
0&\text{ otherwise}
\end{cases}\right) + \O\left(n\sigma_0(n)\frac{4^{\nu(N)}}{k\sqrt N}\right).
\end{split}
\end{equation*}
This proves equation \eqref{trace-2}.
\end{proof}

\subsection{Chebyshev polynomials}\label{Chebyshev}
For $x \in [-2,2],$ let us denote $x = 2\cos \pi t,$ for $t \in [0,1].$  For any integer $n \geq 0,$ the $n$-th Chebyshev polynomial of the second kind is defined as 
$$X_n(x) = \frac{\sin ((n+1)\pi t)}{\sin (\pi t)},\, x = 2\cos \pi t.$$
Thus, $X_0(x) = 1,\,X_1(x) = x,\,X_2(x) = x^2 - 1,\,X_3(x) = x^3 - 2x,$ and so on.
We now recall some classical properties of the Chebyshev polynomials which will be used in this article.
\begin{lemma}\label{key-lemma}
\begin{enumerate}
\item[{\bf (a)}] For any $m \geq n \geq 0,$
$$X_{n}(x)X_{m}(x) = \sum_{i=0}^{n}X_{m-n +2i}(x).$$
\item[{\bf (b)}] For continuous functions $F,\,G$ defined on $[0,1],$
define
$$\left \langle F(t), G(t) \right \rangle = \int_0^1 F(t) G(t) \mu(t) dt.$$
Then we have 
$$\left \langle X_n(2\cos \pi t), 1 \right \rangle =
\begin{cases}
0 &\text{ if }n>0\\
1 &\text{ if }n = 0.
\end{cases}$$

\item[{\bf (c)}] For any $n,m \geq 0,$
$$\left \langle X_n(2\cos \pi t),X_m(2\cos \pi t) \right \rangle =
\begin{cases}
1 &\text{ if }n = m\\
0 &\text{ if }n \neq m.
\end{cases}$$
\item[{\bf (d)}] We have $$\sum_{m=0}^{\infty} \frac{X_{2m}(2\cos \pi t)}{p^m} =  \frac{(p+1)}{(p^{\frac{1}{2}} + p^{-\frac{1}{2}})^2 -4\cos^2\pi t}.$$
\item[{\bf (e)}] For a prime $p,$ define 
$$ \mu_p(t) = \frac{(p+1)}{(p^{\frac{1}{2}} + p^{-\frac{1}{2}})^2 -4\cos^2\pi t}\mu(t).$$
Then,
$$\int_0^1X_n(2\cos \pi t) \mu_p(t) dt
= \begin{cases}
p^{-n/2} &\text{ if }n \text{ is even},\\
0 &\text{ if }n  \text{ is odd}.
\end{cases}$$
\end{enumerate}
\end{lemma}

\begin{proof}
We refer the interested reader to Sections 2.1 and 2.2 of \cite{Serre} for a detailed discussion of the above properties.
\end{proof}

\subsection{Proof of Theorem \ref{main-modular}}\label{modular-proof}
In order to prove Theorem \ref{main-modular}, we have to evaluate, for each $r \geq 1,$ the moments 
$$\lim_{x \to \infty}\frac{1}{|\mathcal F_{N,k}|}\sum_{f \in \mathcal F_{N,k}} \left(\frac{N_{\Phi,L,f}(x) - {\pi}(x)\int_0^1 \phi_L(t) \mu(t)dt}{\sqrt{{\pi}(x)V_{\Phi,L}}}\right)^r.$$
We observe that for sufficiently large values of $x,$
\begin{equation*}
\begin{split}
&N_{\Phi,L,f}(x) - {\pi}(x)\int_0^1 \phi_L(t) \mu(t)dt\\
&= N_{\Phi,L,f}(x) - {\pi}_N(x)\int_0^1 \phi_L(t) \mu(t)dt - \nu(N) \int_0^1 \phi_L(t).
\end{split}
\end{equation*}
Thus, by Proposition \ref{cut-off},
\begin{equation*}
\begin{split}
&N_{\Phi,L,f}(x) - {\pi}(x)\int_0^1 \phi_L(t) \mu(t)dt\\
&= \sum_{m \geq 1}{U}(m)\sum_{  p \leq x \atop { p \nmid N}}a_f(p^{2m}) - \nu(N) \int_0^1 \phi_L(t) + \O\left(x^{-2019}\right).
\end{split}
\end{equation*}
Note here that $M$ is chosen as per the choice of function $\Phi.$ 
The above equation tells us that for a fixed level $N,$
$$\lim_{x \to \infty}\left|\frac{N_{\Phi,L,f}(x) - {\pi}(x)\int_0^1 \phi_L(t) \mu(t)dt - \sum_{m \geq 1}{U}(m)\sum_{  p \leq x \atop { p \nmid N}}a_f(p^{2m})}{\sqrt{\pi(x) V_{\Phi,L}}}\right| = 0.$$
Therefore, computing the moments 
$$\lim_{x \to \infty}\frac{1}{|\mathcal F_{N,k}|}\sum_{f \in \mathcal F_{N,k}} \left(\frac{N_{\Phi,L,f}(x) - {\pi}(x)\int_0^1 \phi_L(t) \mu(t)dt}{\sqrt{{\pi}(x)V_{\Phi,L}}}\right)^r$$
is equivalent to computing 
\begin{equation}\label{finite-sum-moments}
\lim_{x \to \infty}\frac{1}{|\mathcal F_{N,k}|}\sum_{f \in \mathcal F_{N,k}} \left( \frac{\sum_{1 \leq m \leq M} {U}(m)\sum_{  p \leq x \atop {p \nmid N}} a_f(p^{2m})}{\sqrt{\pi(x)V_{\Phi,L}}}\right)^r
\end{equation}
with an appropriate choice of $M = M(x)$.
 %The advantage of Proposition \ref{cut-off} is that it approximates 
%$$N_{\Phi,L,f}(x) - \pi(x)  \int_0^1 \phi_L(t)\mu(t) dt$$
%with finite sums $$\sum\limits_{1\leq m\leq M} {U}(m)\sum_{  p \leq x}a_f(p^{2m})$$
%and
%$$\sum\limits_{1\leq m\leq M} {U}(m)\sum\limits_{\substack{x/2<p\le x \\ p\nmid \Delta(a,b)}} \tilde{a}_{E(a,b)}(p^{2m})$$
 % For these choices of $M = M(x),$ we determine, for every $r \geq 1$, %and
%\begin{equation}\label{finite-sum-moments-e}
%\lim_{x \to \infty} \frac{1}{4AB} \sum\limits_{|a|\le A} \sum\limits_{|b|\le B}  \left(\sum_{1 \leq m \leq M} {U}(m)\sum\limits_{\substack{x/2<p\le x \\ p\nmid \Delta(a,b)}} \tilde{a}_{E(a,b)}\left(p^{2m}\right)\right)^r.
%\end{equation}
Our goal is show that these match the Gaussian $r$-th moments under suitable conditions.

To this end, we prove the following theorem.
%on recalling that $$\frac{1}{L}\cdot U(m) = \frac{1}{L}\left(\widehat{\Phi}\left(\frac{m}{L}\right) -  \widehat{\Phi}\left(\frac{m+1}{L}\right)\right).$$ 
\begin{thm}\label{ModularForms-CLT}
	For any positive integer $r,$
	\begin{align*}
	\frac{1}{|\mathcal F_{N,k}|}\sum_{f \in \mathcal F_{N,k}} \left( \sum_{1 \leq m \leq M}{U}(m)\sum_{  p \leq x \atop {p \nmid N}} a_f(p^{2m})\right)^r &= {\pi(x)^{r/2}}\frac{r!}{2^{r}(r/2)!}\left(\sum_{m\geq 1}{U}(m)^2\right)^{r/2}\delta(r) + o(\pi(x)^{r/2}),
	\end{align*} where 
	$$
	\delta(r)= \begin{cases} 1 & \mbox{ if } r \mbox{ is even,}\\ 0 & \mbox{ if } r \mbox{ is odd.} \end{cases}
	$$
\end{thm}

In order to prove Theorem \ref{ModularForms-CLT}, it would be natural to extend the techniques of \cite{PrabhuSinha}.  In fact, one of the key results in \cite{PrabhuSinha}, namely Theorem 7.5, is a special case of Theorem \ref{ModularForms-CLT} stated above, where $\Phi(t)$ is taken to be a classical function of Beurling \cite{Vaaler}.  In this case, the functions $\phi_{M}(t)$ turn out to be the Beurling-Selberg polynomials and $U(m)$ are exactly the coefficients $\widehat{S}^{\pm}(m)$ considered in \cite{PrabhuSinha}.  Recently, L. Sun, Y. Wen and X. Zhang \cite{SWZ} simplified the techniques of \cite{PrabhuSinha} and presented a much shorter proof of the result pertaining to the higher moments which is essential in proving Theorem \ref{moment}.  In this section, we adapt the techniques of \cite{SWZ} to prove Theorem \ref{ModularForms-CLT}.

We first recall that $a_f(p^{2m}) = X_{2m}(2\cos\pi \theta_f(p))$. 
Define \begin{align*}
Z_M(t) := \sum_{1 \leq m \leq M}{U}(m)X_{2m}(2\cos \pi t). 
\end{align*}
We note the trivial bound 
$$Z_M(\theta_f(p)) = \sum_{1 \leq m \leq M}{U}(m)X_{2m}(2\cos \pi \theta_f(p)) = \sum_{1 \leq m \leq M}{U}(m)a_f(p^{2m})$$ 
satisfies the trivial bound
\begin{equation}\label{Z_Mbound}
Z_M(\theta_f(p)) \ll \sum_{1 \leq m \leq M}m |U(m)| \ll M^2.
\end{equation}
We observe that
\begin{align}
&\frac{1}{|\mathcal F_{N,k}|}\sum_{f \in \mathcal F_{N,k}} \left( \sum_{1 \leq m \leq M}{U}(m)\sum_{  p \leq x \atop {p \nmid N}} a_f(p^{2m})\right)^r \\ \nonumber 
&= \frac{1}{|\mathcal F_{N,k}|}\sum_{f \in \mathcal F_{N,k}}\left(\sum_{  p \leq x \atop {p \nmid N}} Z_M(\theta_f(p))\right)^r \nonumber\\
&= \sum_{u=1}^{r}\sum_{(r_1,\ldots,r_u)}\frac{r!}{r_1!\cdots r_u!}\frac{1}{u!} \sum_{(p_1,\ldots,p_u)} \frac{1}{|\mathcal F_{N,k}|}\sum_{f \in \mathcal F_{N,k}}\prod_{i=1}^{u}Z_M(\theta_f(p_i))^{r_i},\label{multinomial-formula}
	\end{align}
	where
\begin{enumerate}
 \item the sum $ \sum\limits_{(r_1, r_2, \ldots,r_u)}$ is taken over tuples of positive integers 
 $r_1, r_2, \ldots,r_u$ so that\\ $r_1+ r_2+ \cdots + r_u = r,$ that is, over partitions of $r$ into $u$ positive parts and \\
 \item the sum $\sum\limits_{(p_1,p_2,\ldots,p_u)}$ is over $u$-tuples of distinct primes coprime to $N$ and not exceeding $x$.
 \end{enumerate}

We now give a proof of the following proposition, an analogue of which appears in \cite{SWZ}.
\begin{prop}\label{Z_M product}
Let $r_1, r_2, \ldots,r_u$ be positive integers so that\\ $r_1+ r_2+ \cdots + r_u = r$ and let $p_1,p_2,\ldots,p_u$ be $u$ distinct primes coprime to $N.$  Then,
	\begin{align}\label{Z_M power equation}
\frac{1}{|\mathcal{F}_{N,k}|} \sum_{f\in \mathcal F_{N,k}} \prod_{i=1}^{u} Z_M(\theta_f(p_i))^{r_i} & 
= \prod_{i=1}^{u} \int_0^{1} Z_M(t)^{r_i}\mu_{p_i}(t)dt + 
\O_r \left(\frac{4^{\nu(N)}M^{3r}}{k\sqrt{N}}\prod_{i=1}^{u}p_i^{2Mr_i}\right),
	\end{align}
	where $\mu_{p_i}(t)$ is as defined in Lemma \ref{key-lemma}(e).
	%\begin{equation}
	%\mu_{p_i}(\theta)d\theta := (p_i+1) \frac{2\sin^2(\pi \theta) d\theta}{\left(\sqrt{p_i} + \frac{1}{\sqrt{p_i}} \right)^2 -4\cos^2\pi\theta}\\
	%= \frac{(p_i+1)}{\left(\sqrt{p_i} + \frac{1}{\sqrt{p_i}} \right)^2 -4\cos^2\pi\theta} \mu(\theta) d\theta,
	%\end{equation} 
	%the $p$-adic Plancherel measure. 
\end{prop}
\begin{proof}
Expanding $Z_M(\theta_f(p_i))^{r_i}$ using the orthogonality relations among $\{X_n(2\cos\pi t)\}_{n\geq 0}$ given by Lemma \ref{key-lemma}(c), we have 
\begin{align*}
\prod_{i=1}^{u} Z_M(\theta_f(p_i))^{r_i} & = \prod_{i=1}^{u} \left( \sum_{m_i=0}^{Mr_i} \langle Z_M(t)^{r_i}, X_{2m_i}(2\cos \pi t) \rangle X_{2m_i}(2\cos\pi \theta_f(p_i)) \right)\\
&= \sum_{m_1=0}^{r_1}\cdots\sum_{m_u=0}^{r_u} \prod_{i=1}^{u} \langle Z_M(t)^{r_i}, X_{2m_i}(2\cos \pi t) \rangle a_f(p_i^{2m_i}).
\end{align*}
Using Lemma \ref{key-lemma}(e), Proposition \ref{EStrace} and the trivial bound in equation \eqref{Z_Mbound}, we have
\begin{align*}
&\frac{1}{|\mathcal{F}_{N,k}|} \sum_{f\in \mathcal F_{N,k}} \prod_{i=1}^{u} Z_M(\theta_f(p_i))^{r_i}\\ &=\sum_{m_1=0}^{Mr_1}\cdots\sum_{m_u=0}^{Mr_u} \prod_{i=1}^{u} \langle Z_M(t)^{r_i}, X_{2m_i}(2\cos \pi t) \rangle \frac{1}{|\mathcal{F}_{N,k}|} \sum_{f\in \mathcal F_{N,k}}  \prod_{i=1}^{u}a_f(p_i^{2m_i})\\
&= \sum_{m_1=0}^{Mr_1}\cdots\sum_{m_u=0}^{Mr_u} \prod_{i=1}^{u} \left\langle Z_M(t)^{r_i}, X_{2m_i}(2\cos \pi t) \right \rangle {p_i}^{-m_i} \\
& \qquad + \O\left(\frac{4^{\nu(N)}}{k\sqrt{N}} \sum_{m_1=0}^{Mr_1}\cdots\sum_{m_u=0}^{Mr_u} \prod_{i=1}^{u}\left|\left\langle Z_M(t)^{r_i}, X_{2m_i}(2\cos \pi t) \right\rangle\right| p_i^{2m_i}(2m_i)\right)\\
&= \sum_{m_1=0}^{Mr_1}\cdots\sum_{m_u=0}^{Mr_u} \prod_{i=1}^{u} \langle Z_M(t)^{r_i}, X_{2m_i}(2\cos \pi t) \rangle  \int_0^{1}X_{2m_i}(2\cos\pi t)\mu_{p_i}(t)dt \\
& \qquad + \O\left( \frac{4^{\nu(N)}}{k\sqrt{N}}\sum_{m_1=0}^{Mr_1}\cdots\sum_{m_u=0}^{Mr_u} \prod_{i=1}^{u}M^{2r_i}p_i^{2m_i}(2m_i)\right).
\end{align*}
We observe now that
\begin{align*}
&\sum_{m_1=0}^{Mr_1}\cdots\sum_{m_u=0}^{Mr_u} \prod_{i=1}^{u}M^{2r_i}p_i^{2m_i}(2m_i)\\
& \ll M^{2r}2^u\prod_{i=1}^{u}p_i^{2Mr_i}\sum_{m_1=0}^{Mr_1}\cdots\sum_{m_u=0}^{Mr_u}m_1m_2\cdots m_u\\
& \ll M^{2r} 2^u \prod_{i=1}^{u}p_i^{2Mr_i} \prod_{i=1}^{u}\left(Mr_i\right)^2\\
&\ll M^{2r}M^{2u} 2^u \prod_{i=1}^{u}r_i^2 p_i^{2Mr_i} \\
\end{align*}

Thus,
\begin{align*}
&\sum_{m_1=0}^{Mr_1}\cdots\sum_{m_u=0}^{Mr_u} \prod_{i=1}^{u} \langle Z_M(t)^{r_i}, X_{2m_i}(2\cos \pi t) \rangle  \int_0^{1}X_{2m_i}(2\cos\pi t)\mu_{p_i}(t)dt \\
& \qquad + \O\left( \frac{4^{\nu(N)}}{k\sqrt{N}}\sum_{m_1=0}^{Mr_1}\cdots\sum_{m_u=0}^{Mr_u} \prod_{i=1}^{u}M^{2r_i}p_i^{2m_i}(2m_i)\right)\\
&= \sum_{m_1=0}^{Mr_1}\cdots\sum_{m_u=0}^{Mr_u} \prod_{i=1}^{u} \langle Z_M(t)^{r_i}, X_{2m_i}(2\cos \pi t) \rangle  \int_0^{1}X_{2m_i}(2\cos\pi t)\mu_{p_i}(t)dt \\
& \quad + \O\left( \frac{4^{\nu(N)}}{k\sqrt{N}}M^{2r} 2^u M^{2u}\prod_{i=1}^u r_i^2 \prod_{i=1}^u r_i^2 p_i^{2Mr_i}\right)\\
&= \prod_{i=1}^{u} \int_{0}^{1} \sum_{m_i=0}^{Mr_i}\langle Z_M(t)^{r_i}, X_{m_i}(2\cos \pi t) \rangle  X_{2m_i}(2\cos\pi t)\mu_{p_i}(t)d\theta + \O \left(\frac{4^{\nu(N)}2^uM^{4r}}{k\sqrt{N}}\prod_{i=1}^{u}r_i^2p_i^{2Mr_i}\right)\\
&= \prod_{i=1}^u \int_{0}^{1}Z_M(t)^{r_i}\mu_{p_i}(t)dt +  \O_r \left(\frac{4^{\nu(N)}2^u M^{4r}}{k\sqrt{N}}\prod_{i=1}^{u}r_i^2p_i^{2Mr_i}\right).
\end{align*}
\end{proof}
We are now ready to prove Theorem \ref{ModularForms-CLT}.
\begin{proof}
Using Proposition \ref{Z_M product} in \eqref{multinomial-formula}, we observe that
\begin{align*}
&\frac{1}{|\mathcal F_{N,k}|}\sum_{f \in \mathcal F_{N,k}} \left( \sum_{1 \leq m \leq M}{U}(m)\sum_{  p \leq x \atop {p \nmid N}} a_f(p^{2m})\right)^r \\ 
&= \sum_{u=1}^{r}\sum_{(r_1,\ldots,r_u)}\frac{r!}{r_1!\cdots r_u!}\frac{1}{u!} \sum_{(p_1,\ldots,p_u)} \frac{1}{|\mathcal F_{N,k}|}\sum_{f \in \mathcal F_{N,k}}\prod_{i=1}^{u}Z_M(\theta_f(p_i))^{r_i}\\
&=  \sum_{u=1}^{r}\sum_{(r_1,\ldots,r_u)}\frac{r!}{r_1!\cdots r_u!}\frac{1}{u!} \sum_{(p_1,\ldots,p_u)}\prod_{i=1}^{u}\int_0^1 Z_M(t)^{r_i}\mu_{p_i}(t)dt \\
& \quad + \O\left(\sum_{u=1}^{r}\sum_{(r_1,\ldots,r_u)}\frac{r!}{r_1!\cdots r_u!}\frac{1}{u!} \sum_{(p_1,\ldots,p_u)} \frac{4^{\nu(N)}2^uM^{4r}}{k\sqrt{N}}\prod_{i=1}^{u}r_i^2 p_i^{2Mr_i}\right)\\
&= \sum_{u=1}^{r}\sum_{(r_1,\ldots,r_u)}\frac{r!}{r_1!\cdots r_u!}\frac{1}{u!} \sum_{(p_1,\ldots,p_u)}\prod_{i=1}^{u}\int_0^1 Z_M(t)^{r_i}\mu_{p_i}(t)dt\\
&  \quad + \O_r\left(\frac{4^{\nu(N)}}{k \sqrt N} M^{4r} x^{2Mr} \pi(x)^r\right).
\end{align*}
As in Theorem \ref{main-modular} and Proposition \ref{cut-off}, we now consider two cases.

\vspace{.25in}

{\bf Case (a):}  Suppose the Fourier transform $\widehat{\Phi}$ is compactly supported in the interval $[-B,B].$  Then, choose $M \asymp BL.$  The error term equals
$$\O_r\left(\frac{4^{\nu(N)}}{k \sqrt N} M^{4r} x^{2Mr} \pi(x)^r\right) =   \O_r\left(\frac{4^{\nu(N)}}{k \sqrt N}L^{4r}x^{3BLr}\right).$$
Suppose $k = k(x)$ satisfies $\frac{\log k}{\log x} \to \infty$ as $x \to \infty.$  Then,
$$x^{3BLr} = \o(k) \text{ as }x \to \infty$$
and therefore, the error term 
$$\O_r\left(\frac{4^{\nu(N)}}{k \sqrt N} M^{4r} x^{2Mr} \pi(x)^r\right) $$ goes to zero
as $x \to \infty.$

\vspace{.25in}

{\bf Case (b):}  Suppose $\widehat{\Phi}(t) \ll e^{-\lambda |t|^{\omega}}$ for some $\lambda,\omega >0$ as $|t| \to \infty.$  Then, as in Proposition \ref{cut-off}, choose
$$M \asymp L\left(\frac{2020 \lambda}{1 - \epsilon} \log x\right)^{\frac{1}{\omega}}.$$
The error term is
$$\O_r\left(\frac{4^{\nu(N)}}{k \sqrt N} M^{4r} x^{2Mr} \pi(x)^r\right) = \O_r\left(\frac{4^{\nu(N)}}{k \sqrt N} L^{4r}\left(\frac{2020 \lambda}{1 - \epsilon}\log x\right)^{4r/\omega} x^{4rL\left(\frac{2020\lambda}{1 - \epsilon} \log x\right)^{1/\omega}}\right).$$
%If $\frac{\log k}{L(\log x)^{1+1/\omega} \to \infty$ as $x\to\infty$, this contribution is negligible. 
If $\frac{\log k}{(\log x)^{1+1/\omega}} \to \infty$ as $x\to\infty$,
then, for any $r \geq 1,$ 
$$(\log x)^{4r/\omega}x^{4rL\left(\frac{2020\lambda}{1 - \epsilon} \log x\right)^{1/\omega}} = \o(k) \text{ as }x \to \infty.$$
Thus, the error term
$$\O_r\left(\frac{4^{\nu(N)}}{k \sqrt N} M^{4r} x^{2Mr} \pi(x)^r\right) $$
goes to zero as $x \to \infty.$

Therefore, in both cases,
$$\lim_{x \to \infty} \frac{1}{|\mathcal F_{N,k}|}\sum_{f \in \mathcal F_{N,k}} \left( \sum_{1 \leq m \leq M}{U}(m)\sum_{  p \leq x \atop {p \nmid N}} a_f(p^{2m})\right)^r $$
$$ = \lim_{x \to \infty}\sum_{u=1}^{r}\sum_{(r_1,\ldots,r_u)}\frac{r!}{r_1!\cdots r_u!}\frac{1}{u!} \sum_{(p_1,\ldots,p_u)}\prod_{i=1}^{u}\int_0^1 Z_M(t))^{r_i},\mu_{p_i}(t)dt.$$

Next, we estimate the contribution of the main term, by separating the partitions into three types.\\
{\bf Case 1:} If $(r_1,r_2,\ldots,r_u)= (2,2,\ldots,2)$, then we have
\begin{align*}
\prod_{i=1}^{u} \int_0^{\pi} Z_M(t)^{r_i}\mu_{p_i}(t)dt &= \prod_{i=1}^{u} \int_0^{\pi} Z_M(t)^{2}\mu_{p_i}(t)dt.
\end{align*} 
Using the product formula from Lemma \ref{key-lemma} (a), for positive integers $m,n$ with $m\geq n$, we have
%$$X_{m}(2\cos\pi t)X_n(2\cos\pi t)=\sum_{i=0}^{n} X_{m-n+2i}(2\cos\pi t).$$
%Thus,  
\begin{align}
Z_M(t)^2 = 2\sum_{{m,n=1}\atop{m> n}}^{M}{U}(m){U}(n)\sum_{j=0}^{n}X_{2(m-n+j)}(2\cos\pi t) + \sum_{m=1}^M{U}(m)^2\sum_{l=0}^{m}X_{2l}(2\cos\pi t).
\end{align}
Using $\mu(t)= \int_0^{1}2\sin^2\pi t dt$, it is not hard to see that
\begin{equation}\label{mu_p vs mu}
\mu_p(t) = \mu(t) +\O\left(\frac{1}{p}\right).
\end{equation} 
Thus, using equation \eqref{mu_p vs mu} and Lemma \ref{key-lemma}(b), we have  
\begin{align*}
\int_0^{1} Z_M(t)^{2}\mu_{p_i}(t)dt &= \int_0^{1} Z_M(t)^{2}\mu(t)dt + \O\left(\frac{M^2}{p_i}\right)\\
&= 2\sum_{{m,n=1}\atop{m> n}}^{M}{U}(m){U}(n)\sum_{j=0}^{n}\int_0^1 X_{2(m-n+j)}(2\cos\pi t)\mu(t) dt \\
& \quad + \sum_{m=1}^M{U}(m)^2\sum_{l=0}^{m}\int_0^1X_{2l}(2\cos\pi t) \mu(t) dt + \O\left(\frac{M^2}{p_i}\right)\\
&= \sum_{m=1}^M{U}(m)^2 + \O\left(\frac{M^2}{p_i}\right).\\
\end{align*}
By Proposition 3, depending on the choice of $\Phi$, either
$$\sum_{m=1}^{M}{U}(m)^2 = \sum_{m\geq 1}{U}(m)^2$$
 for $M \asymp BL$
 or
 $$\sum_{m=1}^{M}{U}(m)^2 = \sum_{m\geq 1}{U}(m)^2 + \O\left(x^{-2019}\right).$$
for 
$$M \asymp L\left(\frac{2020 \lambda}{1 - \epsilon} \log x\right)^{\frac{1}{\omega}}.$$
In any case, summing over prime tuples $(p_1,\ldots, p_{r/2})$, 
 we obtain the contribution to \eqref{multinomial-formula} from the partition $(r_1,r_2,\ldots,r_u)= (2,2,\ldots,2)$ to be 
 $${\pi(x)^{r/2}}\frac{r!}{2^{r}(r/2)!}\left(\sum_{m\geq 1}{U}(m)^2 + \O\left(x^{-2019}\right)\right)^{r/2} + \O_r\left(\sum_{(p_1,p_2,\dots, p_{r/2})}\prod_{i=1}^{r/2}\frac{M^2}{p_i}\right)$$
$$ =   {\pi(x)^{r/2}}\frac{r!}{2^{r}(r/2)!}\left(\sum_{m\geq 1}{U}(m)^2 \right)^{r/2} + o(\pi(x)^{r/2}).$$ 
\vspace{0.5cm}

{\bf Case 2:} If $r_i=1$ for some $1\leq i\leq u$, without loss of generality we may assume that for some $1 \leq l\leq u$, we have $r_i=1$ for $1\leq i\leq l$ and $r_j \geq 2$ for $j \geq l+1$.  In this case, we have 
\begin{align*}
&\prod_{i=1}^{u} \int_0^{1} Z_M(t)^{r_i}\mu_{p_i}(t)dt\\
&= \prod_{i=1}^{l} \left(\int_0^{1} Z_M(t)\mu_{p_i}(t)dt\right) \prod_{j=l+1}^{u} \int_0^{1} Z_M(t)^{r_j}\mu_{p_j}(t)dt \\
& \ll \prod_{i=1}^{l}\left(\sum_{1 \leq m \leq M}\frac{|{U}(m)|}{p_i^m}\right)\prod_{j=l+1}^{u}{M}^{2r_j},
\end{align*} using Lemma \ref{key-lemma}(e) for the first product and \eqref{Z_Mbound} for the second product. Thus, for such partitions $(r_1,\ldots, r_u)$, \begin{equation*}
\frac{r!}{r_1!\cdots r_u!}\frac{1}{u!} \sum_{(p_1,\ldots,p_u)}\prod_{i=1}^{u} \int_0^{1} Z_M(t)^{r_i}\mu_{p_i}(t)dt \ll_r \pi(x)^{u-l}\left(L(\log x)^{1/\omega}\right)^{2r}(\log\log x)^{l} = o(\pi(x)^{r/2}),
\end{equation*} by observing that for such partitions, $u-l$ is at most $(r-1)/2$. 
\vspace{0.5cm}

{\bf Case 3:} If $r_i\geq2$ for each $i$ and $r_i>2$ for at least one $1\leq i\leq u$. Therefore, $u<\frac{r}{2}$ and we use the trivial estimate in \eqref{Z_Mbound} to write 
\begin{align*}
\prod_{i=1}^{u} \int_0^{1} Z_M(t)^{r_i}\mu_{p_i}(t)d\theta &\ll \pi(x)^{r/2-1}(L(\log x)^{1/\omega})^{2r} =o(\pi(x)^{r/2}).
\end{align*}

Combining all the cases together, we have proved Theorem \ref{ModularForms-CLT}. 

\end{proof} 
By Proposition \ref{variance-exact},
$$\sum_{m\geq 1} {U}(m)^2 = V_{\Phi,L} = \int_0^1\phi_L(t)^2\mu(t)dt - \left( \int_0^1 \phi_L(t) \mu(t) dt \right)^2.$$

Thus, the limits in \eqref{finite-sum-moments} match the Gaussian moments with the appropriate growth conditions on $k = k(x)$ as per the choice of $\Phi$ specified in Theorem \ref{main-modular}.  This completes the proof of Theorem \ref{main-modular}.


\begin{thebibliography}{BLGHT11}

\bibitem[Bir68]{Birch}
B.~J. Birch.
\newblock How the number of points of an elliptic curve over a fixed prime
  field varies.
\newblock {\em J. London Math. Soc.}, 43:57--60, 1968.

\bibitem[BLGHT11]{BGHT}
Tom Barnet-Lamb, David Geraghty, Michael Harris, and Richard Taylor.
\newblock A family of {C}alabi-{Y}au varieties and potential automorphy {II}.
\newblock {\em Publ. Res. Inst. Math. Sci.}, 47(1):29--98, 2011.

\bibitem[BP19]{BP19}
S.~Baier and N.~Prabhu.
\newblock Moments of the error term in the {S}ato-{T}ate law for elliptic
  curves.
\newblock {\em J. Number Theory}, 194:44--82, 2019.

\bibitem[BZ09]{BZ}
S.~Baier and L.~Zhao.
\newblock The {S}ato-{T}ate conjecture on average for small angles.
\newblock {\em Trans. Amer. Math. Soc.}, 361:1811--1832, 2009.

\bibitem[CHT08]{CHT}
L.~Clozel, M.~Harris, and R.~Taylor.
\newblock Automorphy for some $l$-adic lifts of automorphic mod $l$ {G}alois
  representations.
\newblock {\em Publ. Math. Inst. Hautes \'{E}tudes Sci.}, 108:1--181, 2008.

\bibitem[FM96]{FoMu}
{\'E}.~Fouvry and M.~R. Murty.
\newblock On the distribution of supersingular primes.
\newblock {\em Canad. J. Math.}, 48(1), 1996.

\bibitem[HSBT10]{HST}
M.~Harris, N.~Shepherd-Barron, and R.~Taylor.
\newblock A family of {C}alabi-{Y}au varieties and potential automorphy.
\newblock {\em Ann. of Math. (2)}, 171(2):779--813, 2010.

\bibitem[Mon94]{Mont}
Hugh~L. Montgomery.
\newblock {\em Ten lectures on the interface between analytic number theory and
  harmonic analysis}, volume~84 of {\em CBMS Regional Conference Series in
  Mathematics}.
\newblock Published for the Conference Board of the Mathematical Sciences,
  Washington, DC; by the American Mathematical Society, Providence, RI, 1994.

\bibitem[MS10]{MS2}
M.~R. Murty and K.~Sinha.
\newblock Factoring newparts of {J}acobians of certain modular curves.
\newblock {\em Proc. Amer. Math. Soc.}, 138(10):3481--3494, 2010.

\bibitem[Mur85]{VKMurty}
V.~Kumar Murty.
\newblock Explicit formulae and the {L}ang-{T}rotter conjecture.
\newblock {\em Rocky Mountain J. Math.}, 15(2):535--551, 1985.
\newblock Number theory (Winnipeg, Man., 1983).

\bibitem[PS17]{PrabhuSinha}
N.~Prabhu and K.~Sinha.
\newblock Fluctuations in the distribution of {H}ecke eigenvalues about the
  {S}ato-{T}ate measure.
\newblock {\em Int. Math. Res. Not.}, (https://doi.org/10.1093/imrn/rnx238),
  2017.

\bibitem[RT17]{RT}
Jeremy Rouse and Jesse Thorner.
\newblock The explicit {S}ato-{T}ate {C}onjecture and densities pertaining to
  {L}ehmer-type questions.
\newblock {\em Trans. Amer. Math. Soc.}, 369(5):3575--3604, 2017.

\bibitem[Ser97]{Serre}
Jean-Pierre Serre.
\newblock R\'epartition asymptotique des valeurs propres de l'op\'erateur de
  {H}ecke {$T_p$}.
\newblock {\em J. Amer. Math. Soc.}, 10(1):75--102, 1997.

\bibitem[SWZ]{SWZ}
L.~Sun, Y.~Wen, and X.~Zhang.
\newblock Remark on the paper ``{F}luctuations in the distribution of {H}ecke
  eigenvalues about the {S}ato-{T}ate measure".
\newblock
  \url{http://www.yau-awards.science/wp-content/uploads/2018/11/Sun-Lehan-paper_1504.pdf}.

\bibitem[Tay08]{Taylor}
R.~Taylor.
\newblock Automorphy for some $l$-adic lifts of automorphic mod $l$ {G}alois
  representations. ii.
\newblock {\em Publ. Math. Inst. Hautes \'{E}tudes Sci.}, 108:183--239, 2008.

\bibitem[Vaa85]{Vaaler}
J.~D. Vaaler.
\newblock Some extremal functions in {F}ourier analysis.
\newblock {\em Bull. Amer. Math. Soc.}, 12(2):183--216, 1985.

\bibitem[Wan14]{Wang}
Yingnan Wang.
\newblock The quantitative distribution of {H}ecke eigenvalues.
\newblock {\em Bull. Aust. Math. Soc.}, 90(1):28--36, 2014.

\bibitem[Wil95]{Wiles}
Andrew Wiles.
\newblock Modular elliptic curves and {F}ermat's last theorem.
\newblock {\em Ann. of Math. (2)}, 141(3):443--551, 1995.

\end{thebibliography}
\end{document}